\documentclass[10pt]{article}
\usepackage{tikz,amsmath, amssymb,amsthm,color,enumerate,comment,graphicx,float,appendix,enumitem,subcaption,booktabs,arydshln,mathrsfs,hyperref}
\usepackage{xcolor}
\usepackage{algorithm}
\usepackage{algorithmic}
\newcommand{\no}[1]{}  
\renewcommand{\d}{\,\mathrm{d}}
\no{\usepackage{times}\usepackage[subscriptcorrection, slantedGreek, nofontinfo]{mtpro}
\usepackage{tikz}
\usepackage{amsmath, amssymb,amsthm,color,enumerate,comment,graphicx,float,appendix,enumitem,subcaption,booktabs,arydshln,mathrsfs,hyperref}
\usetikzlibrary{arrows.meta, calc}
\renewcommand{\alpha}{\upDelta}}
\graphicspath{{Figures/}}

\title{Stability and Reconstruction of a Nonlinearity in a Parabolic Equation from Partial Boundary Data\thanks{The work of M. Deng is supported by the Hong Kong PhD fellowship scheme, and that of B. Jin is
supported by Hong Kong RGC General Research Fund (Project 14306824), ANR / RGC Joint Research
Scheme (A-CUHK402/24),  NSFC / RGC Joint Research Scheme (N CUHK446/25) and a start-up fund from The Chinese University of Hong Kong. The work of Y. Kian is supported by the French National Research Agency ANR and Hong Kong RGC Joint Research Scheme for the project IdiAnoDiff (grant ANR-24-CE40-7039).}}
\date{}
\author{Jason Choy\thanks{Department of Mathematics, The Chinese University of Hong Kong, Shatin, N.T., Hong Kong (\texttt{zhchoy@math.cuhk.edu.hk, mldeng@link.cuhk.edu.hk, b.jin@cuhk.edu.hk})}\and Maolin Deng\footnotemark[2] \and Bangti Jin\footnotemark[2] \and Yavar Kian\thanks{Univ Rouen Normandie, CNRS, Normandie Univ, LMRS UMR 6085, F-76000 Rouen, France (\texttt{yavar.kian@univ-rouen.fr})}}

\date{\today}
\newtheorem{theorem}{Theorem}[section]

\newtheorem{lemma}{Lemma}[section]

\newtheorem{corollary}{Corollary}[section]

\newcommand{\R}{{\mathbb R}}
\newcommand{\N}{{\mathbb N}}
\newcommand{\RN}[1]{%
  \textup{\uppercase\expandafter{\romannumeral#1}}%
}

\numberwithin{equation}{section}

\renewcommand{\leq}{\leqslant}
\renewcommand{\geq}{\geqslant}

\def\epsilon{\varepsilon}
\def\phi {\varphi}
\def\supp{\text{supp}}
\allowdisplaybreaks
\begin{document}

\maketitle

\begin{abstract}
In this work, we investigate the inverse problem of determining a semilinear term in a nonlinear parabolic equation from one single boundary flux measurement taken on an arbitrary subset of the boundary. More precisely, we address both uniqueness and stability issues of the inverse problem and establish new H\"older-type stability estimates. The H\"older exponent depends explicitly on the measurement configuration as well as on regularity properties of the semilinear term. The analysis relies on a novel approach based on the derivation of a suitable integral identity involving solutions of the associated adjoint equation. This allows reformulating the inverse problem as an inverse source problem with a sign-changing source term. The main results are obtained by combining fundamental properties of parabolic equations, including maximum principle and appropriate energy estimates. Finally, we complement the theoretical analysis with an iterative reconstruction algorithm inspired by inverse source problems, and illustrate its accuracy on several numerical experiments.
\end{abstract}
\section{Introduction}
Let $\Omega \subset \R^d$ for $d \geq 1$ be an open bounded domain with a $\mathcal{C}^{2+\alpha} $ boundary $\partial \Omega$ with $\alpha \in (0,1)$,  $B\in \mathcal{C}^{1+\alpha} (\overline{\Omega}\times [0,T];\mathbb{R}^d)$, and $A(x,t) = (a_{ij}(x,t))_{1\leq i,j\leq d}$, with $a_{ij}\in \mathcal{C}^{1+\alpha}(\overline{\Omega}\times [0,T]) $, be symmetric and elliptic:
\begin{equation}
    \label{eqn:ellipticity}
    a_{ij}(x,t) = a_{ji}(x,t), \quad \xi^T A(x,t) \xi \geq c|\xi|^2, \quad \forall x \in \overline{\Omega}, \,\xi \in \R^d.
\end{equation}
Consider the following semilinear initial-boundary value parabolic problem
\begin{equation}
\label{IBVP}
\left\{\begin{aligned}
    \partial_t u  - \nabla \cdot (A\nabla u ) +{B} \cdot \nabla u + F(u) &= 0, & & \text{in } \Omega \times (0,T) ,\\
    u &=h, & & \mbox{on } \partial\Omega \times (0,T),\\
    u(\cdot,0) &= 0, & & \mbox{in } \Omega,
\end{aligned}\right.
\end{equation}
with the Dirichlet boundary data  $h(x,t)$ taking the product form $\varphi(t) g(x)$. 

In this work,  we investigate the inverse problem of determining the nonlinear reaction term $F$ in \eqref{IBVP} from one single boundary measurement of the flux $\partial_{\nu_A} u$ generated by a single boundary excitation $h$ applied on the lateral boundary $\partial\Omega \times (0,T)$,  
defined by \eqref{pn} below, restricted to an arbitrary open subset $S$ of the boundary $\partial\Omega$. More precisely, we focus our attention on theoretical uniqueness and stability results as well as the numerical reconstruction of the semilinear term $F$ from one single boundary measurement.

Reaction-diffusion equations of the form \eqref{IBVP} are frequently used to describe a wide variety of physical phenomena with applications in chemistry, biology, ecology and physics. Examples include chemical mixing \cite{chadam1994diffusion}, heat conduction (e.g., cooling processes for steel and glass in liquids or solids \cite{KrauseLoch:2002}), phase transition (e.g., crystallization processes of polymers \cite{Capasso:1998} and reaction kinetics of chemicals \cite{Benson:1960}), and population dynamics \cite{Yoshizawa_1970}. 
In these contexts, the concerned inverse problem can be interpreted as the determination of a physical law associated with the nonlinear term $F$ from the measurement of the flux located on a portion of the lateral boundary $\partial\Omega \times (0,T)$.

The identification of the semilinear term in a parabolic equation constitutes a classical inverse problem that has attracted considerable attention over the past few decades. Early contributions to the problem can be traced back to \cite{DR} (see also \cite{Lo,Muzylev:1980} for related developments), where the unique identification of a semilinear term (from boundary measurements generated by both boundary and internal excitations) was proved using the maximum principle and monotonicity argument. This was subsequently extended and refined in \cite{PR}, which provides existence and uniqueness for small time when the semilinear term $F(u)$ is Lipschitz. Reconstruction procedures were also proposed in the one-dimensional case in \cite{PR2} and \cite{CannonDuChateau:1998} using fixed point iteration and the output least-squares formulation, respectively. Similarly, the work \cite{CZ} established the unique recovery of a semilinear term $F(u)$ from a single boundary excitation and measurement by exploiting the maximum principle in conjunction with Hopf's lemma. These results were later generalized by Isakov \cite{isakov1993uniqueness}, who proved the unique determination of a more general nonlinear term of the form $a(x,u)$, with $x \in \Omega$ and $u \in \mathbb{R}$, from an infinite set of boundary measurements, namely the knowledge of the full Cauchy data associated with all admissible solutions. The work \cite{isakov1993uniqueness} made one fundamental methodological breakthrough through the linearization technique, which has since become a central tool in the analysis of inverse problems for nonlinear parabolic equations \cite{choy2026simultaneous,feizmohammadi2022inverse,isakov2001uniqueness,kian2024determining,Kian_Uhlmann_2023}, as well as for broader classes of nonlinear PDEs \cite{KLU,SU}. We would also like to mention  more recent contributions devoted to  uniqueness results based on a single measurement, including boundary observation \cite{homberg2019uniqueness} and final-time measurement \cite{KR1,KR2}.

Despite these important advances, the aforementioned works are predominantly confined to uniqueness results, leaving aside other fundamental issues, e.g., stability estimates and reconstruction procedures. In the one-dimensional case, a  conditional H\"{o}lder stability estimate from the  flux measurement on the full boundary was derived by \cite[Theorem 2]{egger2005global} under a strong \textit{a priori} condition imposed on the solutions of the equation (condition (C) of \cite{egger2005global}).
To the best of our knowledge, in the multidimensional setting, the only result addressing stable recovery of a semilinear term $F(u)$ in a parabolic equation from a single boundary measurement  is due to \cite{choulli2006stable} (see also \cite[Section 3.7]{choulli2009introduction}). The analysis in \cite{choulli2006stable} relies on refined properties of fundamental solutions associated with parabolic operators. However, this result is subject to certain restrictive assumptions, including analyticity conditions imposed on the nonlinear term \cite[Theorem 3]{choulli2006stable} and the analysis holds only for parabolic equations with constant coefficients. Furthermore, the measurement setting is limited to observations on the entire boundary.  In contrast with the elliptic case \cite{kian2024determination}, the literature on numerical reconstruction of a semilinear term in parabolic equations remains scarce. Existing studies \cite{CannonDuChateau:1998,egger2005global} are primarily based on least-squares minimization techniques developed independently of any rigorous uniqueness or stability analysis. This disconnect between theoretical identifiability results and computational schemes significantly restricts the practical applicability of such approaches.

The main objective of the present work is to develop a unified and simplified framework for the stable identification of the semilinear term $F(u)$ from one single boundary measurement. More precisely, fixing an arbitrary open and not empty subset $S$  of the boundary $\partial\Omega$, we establish a H\"older stability estimate on the recovery of the semilinear term $F$ from the partial flux measurement $\partial_{\nu_A} u \mid_{S\times (0,T)}$  associated with a single Dirichlet excitation $h$ (see Theorem \ref{thm:stability} for the precise statement). Moreover, the H\"older exponent  in the stability estimate depends on the location of the measurement set $S$ in accordance with the excitation $h$; and we also show that the exponent may be improved if the semilinear term $F(u)$ enjoys higher regularity (see Corollary \ref{cor:improved-stability} for the precise statement). The approach relies on the derivation of an appropriate integral identity involving solutions to the associated adjoint problem, thereby reducing the identification of the nonlinear term to an inverse source problem. The analysis makes systematic use of the maximum principle and other refined regularity properties of solutions of parabolic equations. We consider general semilinear parabolic equations with variable coefficients and restrict the observation data to an arbitrary subset of the boundary. Moreover, in contrast with existing contributions, the theoretical stability analysis is directly leveraged to design an iterative numerical reconstruction procedure inspired by inverse source methodologies. Numerically, we present an iterative reconstruction method and numerical reconstructions for noisy partial measurements.  

The rest of the paper is organized as follows. In Section \ref{sec:main}, we give a precise formulation of the inverse problem and state the main theoretical results. Then in Sections \ref{sec:uniq-prelim} and \ref{sec:stab-proof}, we present the proofs of the uniqueness and stability of the inverse problem, respectively.  Lastly, in Section \ref{sec:numer} we present numerical results obtained by means of an iterative reconstruction. Throughout, the notation $C$ denotes a generic constant which may change from one line to the other, but is always independent of the semilinear term $F$.

\section{Main results and discussions}\label{sec:main} 

In this section we state the main results of this work. Throughout for any set $X$, we use the standard notation $L^p(X)$, $1 \leq p \leq \infty$ and $\mathcal{C}^{k+r}(X)$, $k \in \N$, $r \in (0,1)$ to respectively represent Lebesgue spaces and H\"older spaces. We use the following function spaces. For $\alpha\in (0,1)$, we use also the parabolic H\"{o}lder space $$\mathcal{C}^{\alpha,\frac\alpha2}(\overline{\Omega}\times[0,T])=\{v\in \mathcal{C}(\overline{\Omega}\times[0,T]):\|v\|_{\mathcal{C}^{\alpha,\frac\alpha2}(\overline{\Omega}\times[0,T])}<\infty\},$$
with the norm $\|\cdot\|_{\mathcal{C}^{\alpha,\frac\alpha2}(\overline{\Omega}\times[0,T])}$ given by
$$
\|v\|_{\mathcal{C}^{\alpha,\frac\alpha2}(\overline{\Omega}\times[0,T])}:=\|v\|_{ \mathcal{C}(\overline{\Omega}\times[0,T])}+\sup_{x,x'\in\overline{\Omega},0\leq t\leq T}\frac{|v(x,t)-v(x',t)|}{|x-x'|^\alpha} + \sup_{0\leq t,t'\leq T,x\in \overline{\Omega}} \frac{|v(x,t)-v(x,t')|}{|t-t'|^{\frac\alpha2}}.
$$
Moreover, we define
\begin{equation*}
    \mathcal{C}^{\alpha+2,\frac{\alpha}{2}+1}(\overline{\Omega}\times[0,T])= \{v\in \mathcal{C}(\overline{\Omega}\times[0,T]):\|v\|_{\mathcal{C}^{\alpha+2,\frac{\alpha}{2}+1}(\overline{\Omega}\times[0,T]) } <\infty\},
\end{equation*}
with the norm $\|\cdot\|_{\mathcal{C}^{\alpha+2,\frac{\alpha}{2}+1}(\overline{\Omega}\times[0,T])}$ given by 
\begin{align*}
    \|v\|_{\mathcal{C}^{\alpha+2,\frac{\alpha}{2}+1}(\overline{\Omega}\times[0,T])} := & \|v\|_{\mathcal{C}(\overline{\Omega}\times[0,T])} +\|\nabla v\|_{\mathcal{C}(\overline{\Omega}\times[0,T])} + \sup_{0\leq t,t'\leq T, x\in\overline{\Omega}}\frac{|\nabla v(x,t)-\nabla v(x,t')|}{|t-t'|^{\frac{\alpha+1}{2}}}\\
   & + \|\partial_t v\|_{\mathcal{C}^{\alpha,\frac{\alpha}{2}}(\overline{\Omega}\times[0,T])} + \sum_{i,j=1}^d\|\partial_{x_i}\partial_{x_j}v\|_{\mathcal{C}^{\alpha,\frac{\alpha}{2}}(\overline{\Omega}\times[0,T])}.
\end{align*}
Similarly we define the H\"{o}lder space $\mathcal{C}^{\alpha+2,\frac{\alpha}{2}+1}(\partial\Omega\times[0,T])$. 

Consider problem \eqref{IBVP} with $0<T<\infty$, and suppose that the advection term ${B}$ lies in $\mathcal{C}^{1+\alpha}(\overline{\Omega}\times [0,T];\mathbb{R}^d)$. Also suppose that the semilinear term $F$ 
 lies in $\mathcal{C}^1 (\R)$ and satisfies 
 \begin{equation}
 \label{asmp. on F}
F(0) = 0, \quad F'(s) \geq 0, \quad s\in \R. 
\end{equation}
We further assume that the spatial boundary data $g$ lies in $\mathcal{C}^{2+\alpha}(\partial\Omega)$ with  $$\|g\|_{\mathcal{C}^{2+\alpha}(\partial\Omega)} \leq \kappa\quad \mbox{and}\quad g(\partial\Omega) = [0,R],$$ 
where $\kappa, R>0$ are fixed constants. The temporal boundary data $\varphi\in \mathcal{C}^\infty([0,T])$ is assumed to take the form 
$$
\varphi(t) = \begin{cases}
    0, & 0\leq t \leq T_1,\\
    f(t), & T_1 \leq t \leq T_2, \\
    1, & T_2 \leq t \leq T,
\end{cases}
$$
where $0\leq T_1 <T_2 \leq T$ with $T_2-T_1 \geq \frac{T}{4}$. The function $f \in \mathcal{C}^\infty([T_1,T_2])$ is assumed to be  strictly increasing on $(T_1,T_2)$ (i.e. $f'(t) > 0$, for $t\in(T_1,T_2)$) and $\|\varphi\|_{\mathcal{C}^{1+\frac\alpha 2}([0,T])}\leq \kappa$. Under these assumptions, problem \eqref{IBVP} admits a unique solution $u\in \mathcal{C}^{2+\alpha,1+\frac \alpha 2}(\overline{\Omega} \times [0,T])$ \cite[Chapter 5, Section 6]{ladyzenskaja_linear_1968}. Lastly, with $\nu = \nu (x)$ being the unit outward normal vector to the boundary $\partial\Omega$ at $x \in \partial\Omega$, we denote by $\partial_{\nu_A}$  the conormal derivative with respect to $A$ defined by
\begin{equation}\label{pn}
\partial_{\nu_A} u(x,t) =A(x,t)\nabla u(x,t)\cdot\nu(x),\quad u\in \mathcal{C}([0,T];\mathcal{C}^1(\overline{\Omega})),\ (x,t)\in\partial\Omega\times [0,T].
\end{equation}

We start with a uniqueness result that can be stated as follows. 
\begin{theorem}
\label{thm:uniqueness(new)}
    For $j=1,2$, let $F_j \in \mathcal{C}^1(\R)$ satisfy \eqref{asmp. on F}, and let $u_j$ be the solution to problem \eqref{IBVP} with $F = F_j$. 
 Let $S$ be an arbitrary open and non-empty subset of the boundary $\partial\Omega$. Suppose 
    \begin{equation}
    \label{cond. for uniqueness}
    \partial_{\nu_A} u_1(x,t) = \partial_{\nu_A} u_2 (x,t), \quad (x,t) \in S\times (0,T).
    \end{equation}
    Then, one of the two following conditions holds:
    \begin{enumerate}
        \item[{\rm(i)}] The function $F = F_2-F_1 $ changes sign an infinite number of times on $[0,R]$.
        \item[{\rm(ii)}] $F_1 \equiv F_2$ on $[0, R]$. 
    \end{enumerate} 
\end{theorem}
Note that the uniqueness result is unconditional in the sense that it requires neither \textit{a priori} bounds about the semilinear term $F(u)$ nor restrictions on the measurement set $S$. This stands in contrast to the conditional H\"older stability result for $F$ stated below, where an \textit{a priori} bound on $F$ is required, and both the H\"older exponent and stability constant $C$ depend on $S$.

\begin{theorem}
\label{thm:stability}
For $j=1,2$, let $F_j \in \mathcal{C}^1(\R)$ satisfy \eqref{asmp. on F} and let $u_j$ be the solution to problem \eqref{IBVP} with $F = F_j$. Suppose that $F = F_2 - F_1$ changes sign at most $N$ times on $[0,R]$ and that there exists $M>0$ such that $\| F_j\|_{\mathcal{C}^1([0,R])} \leq M.$
Let $S$ be an arbitrary non-empty open subset of the boundary $\partial\Omega$. If there exists $x' \in S$ such that $g(x') = R$, then there holds 
    \begin{equation}
        \label{stab. est. for x' in S}
\sup_{s\in [0,R]} |F_1(s)-F_2(s)| \leq C\sup_{S \times (0,T)} |\partial_{\nu_A} u_1 - \partial_{\nu_A} u_2|^{\frac{1}{(d+2)^{N+1}}},
\end{equation}
with $C = C(\Omega,T,A,{B},M,N,\kappa,d^*)>0$, with 
\begin{align}\label{eqn:d-ast}
d^* = \min_{\substack{x \in S \\ g(x)=R}} d(x,\partial S) .
\end{align}
Otherwise, with $C = C(\Omega,T,A,{B},M,N,\kappa,|S|)>0$, there holds
\begin{equation}
        \label{stab. est. for x' not in S}
\sup_{s\in [0,R]} |F_1(s)-F_2(s)| \leq C \left( \sup_{S \times (0,T)} |\partial_{\nu_A} u_1 - \partial_{\nu_A} u_2| \right)^{\frac{1}{(d+3)^{N+1}}}.
\end{equation}
\end{theorem}

Finally, we improve the H\"older exponent appearing in Theorem \ref{thm:stability} under stronger regularity assumptions on the semilinear term $F(u)$. 
\begin{corollary}
\label{cor:improved-stability}
For $j=1,2$, let $F_j \in \mathcal{C}^2(\R)$ satisfy \eqref{asmp. on F} with $\|F_j\|_{\mathcal{C}^2([0,R])} \leq M$ for some $M > 0$, and let $u_j$ be the solution to problem \eqref{IBVP} with $F = F_j$. Suppose that $F = F_2 - F_1$ changes sign at most $N$ times on $[0,R]$, and let $S$ be an arbitrary non-empty open subset of the boundary $\partial\Omega$. If there exists $x' \in S$ such that $g(x') = R$, then there holds 
\begin{equation}
\label{x in S, increased regularity on F}
\sup_{s \in [0,R]} |F_1(s) - F_2(s)| \leq C \sup_{S \times (0,T)} |\partial_{\nu_A} u_1 - \partial_{\nu_A} u_2| ^{\left(\frac{2}{d+3}\right)^{N+1}},
\end{equation}
with $C = C(\Omega, T, A,{B}, M,N, \kappa, d^*) > 0$, with $
d^*$ defined in \eqref{eqn:d-ast}. 
Otherwise, there holds 
\begin{equation}
\label{x not in S, increased regularity on F}
\sup_{s \in [0,R]} |F_1(s) - F_2(s)| \leq C \left( \sup_{S \times (0,T)} |\partial_{\nu_A} u_1 - \partial_{\nu_A} u_2| \right)^{\left(\frac{2}{d+4}\right)^{N+1}},
\end{equation}
with $C = C(\Omega, T, A,{B}, M,N, \kappa,|S|) > 0$.
\end{corollary}

To the best of our knowledge, Theorem~\ref{thm:stability} provides the first stability result for the determination of a semilinear term \( F(u) \) in a parabolic equation from partial boundary measurements. Specifically, the measurement $\partial_{\nu_A} u$ is taken on an arbitrary nonempty open subset \( S \subset \partial\Omega \), and the influence of the location of \( S \) on the Hölder exponent of the stability estimates \eqref{stab. est. for x' in S} and \eqref{stab. est. for x' not in S} is also analyzed. The estimate \eqref{stab. est. for x' in S} relaxes the regularity requirement in \cite[Theorem 2]{choulli2006stable} from $F_j\in C^2([0,R])$ to $F_j\in C^1([0,R])$. Furthermore, Corollary~\ref{cor:improved-stability} quantifies the dependence of the exponent on the regularity assumed for the nonlinearity. Note that restricting the measurement to a portion $S$ of the boundary $\partial\Omega$ can be useful for various practical applications. In particular, the case of data on a disjoint set \( S \cap \textrm{supp}(g) = \emptyset \) may be required in practical applications, for which the overlap between excitations and measurements cannot be allowed.
In addition to restricting the flux data to an arbitrary portion $S$ of the boundary $\partial\Omega$, in contrast to \cite{choulli2006stable}, the results of Theorem~\ref{thm:stability} and Corollary \ref{cor:improved-stability} are formulated for a general class of parabolic equations with variable coefficients depending on both space and time. Moreover, the admissible semilinear terms \( F(u) \) are not subject to any analyticity condition as in \cite[Theorem 3]{choulli2006stable}.

The proof of Theorem \ref{thm:stability} relies crucially on the integral identity \eqref{main equality} associated with the adjoint equation, combined with techniques from inverse source problems to recover the nonlinearity $F(u)$. This is complemented by refined properties of solutions to parabolic equations, including  the maximum principle and suitable energy estimates, collected in Section \ref{sec:uniq-prelim}. One key point is the flexibility in the choice of solutions to the adjoint problem, which allows us to work with one single partial boundary measurement. This contrasts sharply with the approach in \cite{choulli2006stable}, which relies on Gaussian lower bounds of the fundamental solution of parabolic equations with a Neumann boundary condition.
Moreover, we derive a numerical reconstruction of the nonlinearity $F(u)$ following the same strategy. To the best of our knowledge, this is the first work combining stability analysis with numerical reconstruction for a semilinear term $F(u)$ in a multidimensional parabolic equation.

%, to the existing literature devoted to the determination of semilinear terms from parabolic problems (see for example \cite{choulli2006stable}), which does not employ an adjoint formulation and instead relies on properties of the fundamental solution of \eqref{IBVP}, which inherently requires full boundary data.

%We briefly summarize the overall proof strategies for the theoretical results. The proof of the uniqueness and stability results both rely on consideration of a suitable adjoint problem \eqref{defn of v}. Inspired by the work \cite{kian2024determination} which utilizes a similar adjoint approach for a semilinear elliptic problem, we first derive a fundamental integral identity \eqref{main equality} which relates the solution of the forward problem \eqref{IBVP} to the solution of the adjoint problem. Then, we derive both results from the integral identity by utilizing properties of solutions to parabolic problems as well as by a careful choice of boundary data for the adjoint problem. We note that the freedom in selecting this boundary data is precisely what enables us to work with only partial boundary measurements. This is in contrast, to the existing literature devoted to the determination of semilinear terms from parabolic problems (see for example \cite{choulli2006stable}), which does not employ an adjoint formulation and instead relies on properties of the fundamental solution of \eqref{IBVP}, which inherently requires full boundary data.

\section{Proof of uniqueness and preliminaries}
\label{sec:uniq-prelim}
This section is dedicated to the proof of the uniqueness result stated in Theorem \ref{thm:uniqueness(new)} and the derivation of different properties of solutions of parabolic equations which will be key ingredients in the proof of the stability estimates in Theorem~\ref{thm:stability}.
We introduce two related problems which will be crucial for the proof of both results. First, for $j=1,2$, let $u_j$ be the solution of the problem \eqref{IBVP} with $F = F_j$. Then,  $u = u_1-u_2$ solves the IBVP
\begin{equation}
\label{eqn solved by u}
\left\{\begin{aligned}
    \partial_t u  - \nabla \cdot (A \nabla u ) + {B}\cdot \nabla u + qu &= F(u_2), && \text{in } \Omega \times (0,T) ,\\
    u &= 0, && \text{on } \partial\Omega \times (0,T) ,\\
    u(\cdot,0) &= 0,&& \mbox{in }\Omega,
\end{aligned}\right.
\end{equation}
where $F=F_2-F_1$ and the term $q \in \mathcal{C}^1(\overline{\Omega}\times [0,T])$ is given by 
$$
q(x,t)=\int_0^1 F_1'(su_1(x,t) + (1-s)u_2(x,t)) \d s.
$$
Fix $0<\widetilde{T}<T$ and let $v$ be the solution of the adjoint problem 
\begin{equation}
\left\{\begin{aligned}
\label{defn of v}
    -\partial_t v -\nabla \cdot (A \nabla v) - \nabla \cdot ({B}v) + qv &= 0,& & \text{in } \Omega \times (0,\widetilde{T}), \\
    v(x,t) &= h, & &\text{on }  \partial\Omega \times (0,\widetilde{T}), \\
    v(\cdot,\widetilde{T}) &= 0,&  &  \mbox{in }\Omega,
\end{aligned}\right.
\end{equation}
where the boundary data $h \in \mathcal{C}^\infty(\partial\Omega \times [0,\widetilde{T}])$ is nonnegative on $\partial\Omega \times (0,\widetilde{T})$ and 
should satisfy the compatibility conditions $h(x,\widetilde{T}) = \partial_th(x,\widetilde{T})= 0$ for $x\in \partial\Omega$. 
The precise choice of $h$ will change below. Now, we can prove Theorem \ref{thm:uniqueness(new)}.

\begin{proof}[Proof of Theorem \ref{thm:uniqueness(new)}] We prove this result by contradiction. Namely, let $F = F_2 - F_1$ and suppose that $F$ changes sign a finite number of times on $[0, R]$ but $F\not\equiv 0$ on $[0, R]$. Without loss of generality, we can find $r$ and $s$ with $0<r<s<1$ such that
    $$
    F(rR) > 0\quad\mbox{and} \quad F \geq 0 \text{ on } [0,sR] .
    $$
    Set $\widetilde{T} = f^{-1} (s)$. We choose the boundary data $h \in \mathcal{C}^\infty(\partial\Omega \times [0,\widetilde{T}])$ to be a bump function satisfying 
    $$
    \supp(h) = \overline{S} \times [0,\widetilde{T}] \quad \text{and} \quad h>0 \text{ on } S\times (0,\widetilde{T}).
    $$
Let $v$ be the solution of the adjoint problem \eqref{defn of v}. Multiplying equation \eqref{eqn solved by u} by $v$, integrating over the domain $\Omega$ and integrating by parts yield
    \begin{equation}
\label{main equality}
\int_0^{\widetilde{T}} \int_\Omega
 F(u_2(x,t))v(x,t)  \d x\d t = \int_0^{\widetilde{T} }\int_{\partial\Omega} (\partial_{\nu_A} u_1 - \partial_{\nu_A} u_2) h(x,t) \d \sigma(x) \d t.\footnote{The identity \eqref{main equality} relates  the variation $F$ of the semilinear terms to the change in the measured flux.}
\end{equation}
By the support condition of $h$ and condition \eqref{cond. for uniqueness}, the identity \eqref{main equality} implies
\begin{equation}
\label{int F(u_2)v = 0}
\int_0^{\widetilde{T}} \int_\Omega F(u_2(x,t))v(x,t) \d x\d t = 0. 
\end{equation}
By applying the weak maximum principle to $u_2$ (see, e.g., \cite[p. 39]{friedman2008partial} and \cite[Theorem 1.46]{choulli2009introduction}), we derive that for $(x,t) \in \overline{\Omega}\times [0,\widetilde{T}]$, there holds 
$$
0 \leq u_2(x,t) \leq \max_{(x,t) \in \partial\Omega \times [0,\widetilde{T}]} \varphi(t)g(x) =sR,
$$
and thus by Assumption \eqref{asmp. on F}, we have $F(u_2) \geq 0$ on $\overline{\Omega}\times [0,\widetilde{T}]$. The strong maximum principle, applied to the function $\widetilde{v}(x,t) = v(x,\widetilde{T}-t)$, implies
$$
v(x,t) >0, \quad (x,t) \in \Omega \times (0,\widetilde{T}). 
$$
Hence, we may deduce from \eqref{int F(u_2)v = 0} that $F(u_2) \equiv 0$ on $\Omega \times (0,\widetilde{T})$. Moreover, by the continuity of the map $(x,t) \mapsto F(u_2(x,t))$, we can deduce $F(u_2) \equiv 0$ on $\overline{\Omega}\times [0,\widetilde{T}]$. Now, pick an arbitrary point $x_0 \in \partial\Omega$ with $g(x_0)=R$ and set $t_0 = f^{-1}(r)$. Then,
$$
F(u_2(x_0,t_0)) = F(\varphi(t_0)g(x_0)) = F(rR) >0, 
$$
which clearly contradicts the assertion $F(u_2) \equiv 0$. This concludes the proof of Theorem \ref{thm:uniqueness(new)}. 
\end{proof} 

Next, we show several properties of the solution $v$ of problem \eqref{defn of v} that will be required for deriving the stability estimates stated in Theorem \ref{thm:stability}.
\begin{lemma}
\label{lem: v_M leq v leq v_0}
Suppose that $\|F\|_{\mathcal{C}^1([0,R])} \leq M$, and let $v_\tau$ for $\tau = 0,M$ be the solution to the auxiliary problem 
$$
\left\{\begin{aligned}
    -\partial_t v_\tau -\nabla \cdot (A \nabla v_\tau) - \nabla \cdot ({B}v_\tau) + \tau v_\tau &= 0,& & \text{in } \Omega \times (0,\widetilde{T}), \\
    v_\tau&= h &&\text{on }  \partial\Omega \times (0,\widetilde{T}), \\
    v_\tau(\cdot,\widetilde{T}) &= 0,  && \mbox{in }\Omega.
\end{aligned}\right.
$$
Then, we have 
$$
v_M(x,t) \leq v(x,t) \leq v_0(x,t), \quad (x,t) \in \Omega \times [0,\widetilde{T}].
$$
\end{lemma}
\begin{proof}
By the construction of $q$, we deduce $\|q\|_{\mathcal{C}(\overline{\Omega}\times [0,T])} \leq \|F\|_{\mathcal{C}^1([0,R])} \leq M$. For $\tau=0$, note that the function $v_0-v$ satisfies
$$
-\partial_t (v_0-v) - \nabla \cdot (A\nabla (v_0-v)) - \nabla \cdot ({B}(v_0-v)) = qv, \quad \text{in } \Omega \times [0,\widetilde{T}]. 
$$
Now, since $q(x,t)v(x,t) \geq 0 $ on $\Omega \times (0,\widetilde{T})$ and $v_0 - v \equiv 0$ on $\partial\Omega \times (0,\widetilde{T})$, by the weak maximum principle, we deduce $$\min_{(x,t) \in \overline{\Omega} \times [0,\widetilde{T}]} v_0(x,t)- v(x,t) = 0.
$$
Then it follows directly that $v(x,t) \leq v_0(x,t)$, $(x,t) \in \Omega \times [0,\widetilde{T}]$. The assertion $v\geq v_M$ can be deduced analogously by considering $v-v_M$.
\end{proof}

\begin{lemma}
\label{lem:schauder-type est} Suppose that $\|F\|_{\mathcal{C}^1([0,R])}\leq M$. Then, the following three statements hold.
\begin{enumerate}  \item[{\rm(i)}] Let $u$ be the unique solution of problem \eqref{IBVP}. Then there exists $C_1 = C_1(\Omega,T,A,{B},M,\kappa)>0$ such that
$$
\|u\|_{\mathcal{C}^{2+\alpha,1+\frac\alpha 2}(\overline{\Omega}\times[0,T])} \leq C_1.
$$
\item[{\rm(ii)}] Suppose that $h \in \mathcal{C}^\infty(\partial{\Omega}\times [0,\widetilde{T}])$ satisfies the compatibility condition $h(\cdot,\tilde{T}) = 0$ on the boundary $\partial\Omega$ and let $v$ be the unique solution of problem \eqref{defn of v}. Then there exists $C_2 = C_2(\Omega,T,A,{B},M) >0$ such that 
$$
\|v\|_{\mathcal{C}^{2+\alpha,1+\frac \alpha 2}(\overline{\Omega}\times [0,\widetilde{T}])} \leq C_2 \|h\|_{\mathcal{C}^{2+\alpha,1+\frac\alpha 2}(\partial\Omega\times [0,\widetilde{T}])}.
$$
\item[{\rm(iii)}] Suppose that $h \in \mathcal{C}^\infty(\partial{\Omega}\times [0,\widetilde{T}])$ satisfies the compatibility condition $h(\cdot,\tilde{T}) = 0$ on the boundary $\partial\Omega$ and let $v$ be the unique solution of problem \eqref{defn of v}. Then there exists $C_3 = C_3(\Omega,T,A,{B},M) >0$ such that 
$$
\|v\|_{L^1(\Omega\times (0,\widetilde{T}))} \leq C_3 \|h\|_{L^1(\partial\Omega\times (0,\widetilde{T}))}.
$$
\end{enumerate}
\end{lemma}
\begin{proof}
We prove the three estimates separately. For {(i)}, let $\Phi$ be the solution of the following linear problem
$$
\left\{\begin{aligned}
      \partial_t\Phi - \nabla \cdot (A\nabla \Phi) + {B} \cdot \nabla \Phi &= 0, & & \text{in } \Omega \times (0,T),\\
      \Phi &= \varphi g, & & \text{on } \partial\Omega \times (0,T),\\
      \Phi(\cdot,0) &= 0, && \mbox{in }\Omega.
  \end{aligned}\right.
  $$
By the classical H\"older \textit{a priori} regularity estimate \cite[Chapter 4, Theorem 5.2]{ladyzenskaja_linear_1968}, we have 
$$
\|\Phi\|_{\mathcal{C}^{2+\alpha,1+\frac\alpha2}(\overline{\Omega}\times [0,T])} \leq C\|\varphi g\|_{\mathcal{C}^{2+\alpha,1+\frac\alpha 2}(\partial\Omega\times [0,T])} \leq C,
$$
  with $C = C(\Omega,T,A,{B},\kappa)>0$. Then, the function $w = u - \Phi$ solves
  $$
\left\{\begin{aligned}
      \partial_t w - \nabla \cdot (A\nabla w) + {B}\cdot \nabla w &= F_w, & & \text{in } \Omega \times (0,T),\\
      w &= 0, & & \text{on } \partial\Omega \times (0,T),\\
      w(\cdot,0) &= 0, & & x \in \Omega,
  \end{aligned}\right.
  $$
with $F_w=-F(u) - \partial_t\Phi + \nabla \cdot (A\nabla \Phi)- {B} \cdot \nabla \Phi$.  By the  H\"older \textit{a priori} regularity estimate \cite[Theorem 4.1]{friedman2008partial}, we have for any $\delta \in (0,1)$, 
  $$
  \|w\|_{\mathcal{C}^{1+\delta, \frac{1+\delta}{2}}(\overline{\Omega}\times [0,T])} \leq C\|F_w\|_{\mathcal{C}(\overline{\Omega}\times[0,T]}= C\|-F(u) - \partial_t\Phi + \nabla \cdot (A\nabla \Phi)- {B} \cdot \nabla \Phi \|_{\mathcal{C}(\overline{\Omega}\times [0,T])}.
  $$
By the weak maximum principle, we derive that $0\leq u(x,t) \leq R$ on $\overline{\Omega}\times [0,T]$. By combining these two estimates, it follows that
\begin{equation}
\label{C1+delta, delta est. for u}
\begin{aligned}
\|u\|_{\mathcal{C}^{1+\delta, \frac{1+\delta}{2}}(\overline{\Omega}\times [0,T])} &\leq \|w\|_{\mathcal{C}^{1+\delta, \frac{1+\delta}{2}}(\overline{\Omega}\times [0,T])} + \|\Phi\|_{\mathcal{C}^{1+\delta, \frac{1+\delta}{2}}(\overline{\Omega}\times [0,T])}\\
& \leq C(1+\|F(u)\|_{\mathcal{C}(\overline{\Omega}\times [0,T])} + \|\Phi\|_{\mathcal{C}^{2+\alpha,1+\frac\alpha 2}(\overline{\Omega}\times [0,T])}
)\\
&\leq C(1+ \|F\|_{\mathcal{C}([0,R])} ) \leq C.
\end{aligned}
\end{equation}
Now, by regarding $u$ as the solution to the linear problem
\begin{equation*}
\left\{\begin{aligned}
    \partial_t \zeta  - \nabla \cdot (A \nabla \zeta ) + {B} \cdot \nabla \zeta &= -F(u), & & \text{in } \Omega \times (0,T) ,\\
    \zeta &= \varphi g, && \text{on } \partial\Omega \times (0,T),\\
    \zeta(\cdot,0) &= 0, & & \mbox{in }\Omega,
\end{aligned}\right.
\end{equation*}
we may apply again the classical H\"older \textit{a priori} estimate and conclude 
$$
\|u\|_{\mathcal{C}^{2+\alpha,1+\frac \alpha 2}(\overline{\Omega}\times [0,T])}\leq C (\|F(u)\|_{C^{\alpha,\frac \alpha 2}(\overline{\Omega}\times [0,T])} + \|\varphi g\|_{\mathcal{C}^{2+\alpha,1+\frac\alpha 2}(\partial\Omega\times [0,T])} ).
$$
Next, since $F(0) = 0$, cf. condition \eqref{asmp. on F}, we have \begin{equation}
\label{eqn1}
\|F(u)\|_{\mathcal{C}^{\alpha,\frac \alpha 2}(\overline{\Omega}\times [0,T])} \leq \|F\|_{\mathcal{C}^1([0,R])} \|u\|_{\mathcal{C}^{\alpha,\frac \alpha 2}(\overline{\Omega}\times [0,T])}.
\end{equation}
Hence, the assertion (i) follows by taking $\delta = \alpha$ in the estimate \eqref{C1+delta, delta est. for u} and combining with the estimate \eqref{eqn1}.

The assertion {(ii)} follows directly from the observation $\|q\|_{\mathcal{C}(\overline{\Omega}\times [0,T])} \leq \|F\|_{\mathcal{C}^1([0,R])} \leq M$ as well as the classical H\"older \textit{a priori} estimate.

Last, for {(iii)}, let $U(y,\tau,x,t)$, $(x,y,t)\in\overline{\Omega}\times\overline{\Omega}\times(0,\widetilde{T}),\ \tau\in(t,\widetilde{T})$, be the fundamental solution to the backward problem
$$
\left\{\begin{aligned}
      -\partial_t v -\mathcal{L}v &= 0, & & \text{in } \Omega \times (0,\widetilde{T}), \\
    v &= 0, & &\text{on }  \partial\Omega \times (0,\widetilde{T}), \\
    v(\cdot,\widetilde{T}) &= 0,&  &\mbox{in }\Omega,
\end{aligned}\right.
$$
with $\mathcal{L}v = \nabla \cdot (A \nabla v) + \nabla \cdot ({B}v)$ (see \cite[Section 7]{ito1992diffusion} for more details). We fix $\partial_{\nu_y} U$ defined by
$$\partial_{\nu_y} U(y,\tau,x,t)=\nabla_y U(y,\tau,x,t)\cdot\nu(y), \quad (x,y,t)\in\Omega\times\partial\Omega\times(0,\widetilde{T}),\ \tau\in(t,\widetilde{T}).$$
In view of \cite[p. 53]{ito1992diffusion}, the function $\partial_{\nu_y} U$ satisfies the following estimate 
\begin{equation}\label{eqn:est-fund}
\int_\Omega -\partial_{\nu_y} U(y,\tau,x,t) \d \sigma(y) \leq C(\tau-t)^{-\frac12}\exp(C(\tau-t)),\quad (x,t)\in\Omega\times(0,\widetilde{T}),\ \tau\in(t,\widetilde{T}), 
\end{equation}
for some $C=C(\Omega,T,A,{B})>0$.  Next, let $v_0$ be defined as in Lemma \ref{lem: v_M leq v leq v_0}. By \cite[Theorem 9.2*]{ito1992diffusion}, $v_0$ admits the following representation
$$
v_0(x,t) = \int_t^{\widetilde{T}} \int_{\partial\Omega} - \partial_{\nu_y} U (y,\tau,x,t) h(y,\tau) \d \sigma(y) \d\tau,\quad (x,t)\in\Omega\times(0,\widetilde{T}).
$$
Hence, by the positivity of $v$, Lemma \ref{lem: v_M leq v leq v_0}, and also applying the estimate \eqref{eqn:est-fund} for $U$, we have
\begin{align*}
\|v\|_{L^1(\Omega \times (0,\widetilde{T}))} &= \int_0^{\widetilde{T}}\int_\Omega v(x,t) \d x\d t \leq \int_0^{\widetilde{T}}\int_\Omega v_0(x,t) \d x\d t \\
& \leq \int_0^{\widetilde{T}}\int_\Omega\int_t^{\widetilde{T}} \int_{\partial\Omega} - \partial_{\nu_y} U (y,\tau,x,t) h(y,\tau) \d \sigma(y) \d\tau \d x\d t \\
& \leq \int_0^{\widetilde{T}} \int_{\partial\Omega} h(y,\tau) \int_0^\tau \int_{\Omega} - \partial_{\nu_y} U (y,\tau,x,t) \d x\d t\d \sigma(y)\d\tau \\
& \leq \int_0^{\widetilde{T}} \int_{\partial\Omega} h(y,\tau) \int_0^\tau C(\tau-t)^{-\frac 12} \exp(C(\tau-t)) \d t\d \sigma(y)\d\tau \\
& \leq \int_0^{\widetilde{T}} \int_{\partial\Omega} Ch(y,\tau) \int_0^{\widetilde{T}} \xi^{-\frac 12} \exp(C\xi) \d\xi \d \sigma(y) \d\tau \\
& \leq C_3 \|h\|_{L^1(\partial\Omega \times (0,\widetilde{T}))}.
 \end{align*}
This completes the proof of the assertion {(iii)} and also the lemma.
\end{proof}

\section{Proof of stability estimates}
\label{sec:stab-proof}

Using Lemmas \ref{lem: v_M leq v leq v_0} and \ref{lem:schauder-type est}, we can prove Theorem \ref{thm:stability}.

\begin{proof}[Proof of Theorem \ref{thm:stability}]

Following the proof of Theorem \ref{thm:uniqueness(new)}, let $F = F_2-F_1$ and let $v$ be the solution to problem 
\eqref{defn of v}, with $\widetilde{T}$ and $h$ to be fixed below. Let  $s_1,\ldots,s_N$ be the $N$ points where $F$ changes sign, and denote $s_0 = 0$ and $s_{N+1} = R$. We define
$$|F(r_j)| = \sup_{s\in [s_j,s_{j+1}]}|F(s)|, \quad t_j = f^{-1}(r_jR^{-1}),\quad  \widetilde{T}_j =f^{-1}(s_{j+1}R^{-1}).$$ 
First we fix  $\widetilde{T} = \widetilde{T}_0$ and assume without loss of generality that $F \geq 0 $ on $[s_0,s_1]$. By the weak maximum principle applied to $u_2$, we derive the following estimate
$$
0= s_0 = \min_{(x,t) \in \partial{\Omega}\times [0,\widetilde{T}_0]}\varphi(t)g(x) \leq u_2(x,t) \leq \max_{(x,t)\in \partial\Omega \times [0,\widetilde{T}_0]} \varphi(t)g(x) = s_1,\quad (x,t)\in \overline{\Omega}\times [0, \widetilde{T}_0],
$$
which implies $F(u_2(x,t)) \geq 0$, $(x,t)\in\Omega \times [0,\widetilde{T}_0]$. Now, let 
$$0<\varepsilon <m_0:=\min\{t_0,\widetilde{T}_0-t_0\}$$ 
be arbitrary, and let $x' \in \partial\Omega$ be any point such that $g(x') = R$. Then, consider an $\epsilon$ - neighborhood of the point $(x',t_0)$:
$$
\mathcal{O}_{x',\epsilon} \times \mathcal{T}_{0,\varepsilon} =  \{x \in \overline{\Omega}:|x-x' |<\varepsilon\} \times (t_0-\varepsilon,t_0+\varepsilon),
$$
where $|\cdot|$ denotes the Euclidean norm of vectors in $\mathbb{R}^d$. Note that $\ \mathcal{T}_{0,\varepsilon}\subset [0,\widetilde{T}_0]$ by construction. Now, if $(x,t) \in  \mathcal{O}_{x',\epsilon} \times \mathcal{T}_{0,\varepsilon}$, then by the triangle inequality and Lemma   \ref{lem:schauder-type est} (i), we have
\begin{equation}
    \begin{aligned}
F(u_2(x,t)) &\geq F(u_2(x',t_0)) - |F(u_2(x',t_0))-F(u_2(x,t))| \\
&\geq F(\varphi(t_0)g(x')) - \|F\|_{\mathcal{C}^1([0,R])} \|u_2\|_{\mathcal{C}^1(\overline{\Omega}\times[0,T])} |(x',t_0)-(x,t)| \\
&\geq F(r_0) - \sqrt{2}M C_1\varepsilon.
\end{aligned}\label{eqn: F>=F(r_0) - C*eps}
\end{equation}
Hence, by the nonnegativity of $F(u_2)$ and $v$ on $\overline{\Omega}\times [0,\widetilde{T}_0]$, we have 
\begin{align}
 \int_0^{\widetilde{T}_0}\int_\Omega F(u_2(x,t)) v(x,t) \d x \d t &\geq\int_{\mathcal{T}_{0,\varepsilon}}\int_{\mathcal{O}_{x',\varepsilon}} F(u_2(x,t)) v(x,t) \d x \d t \nonumber\\
 &\geq (F(r_0) - \sqrt{2}MC_1\varepsilon) \int_{\mathcal{T}_{0,\varepsilon}} \int_{\mathcal{O}_{x',\varepsilon}} v(x,t) \d x\d t. \label{eqn2}
\end{align}
Meanwhile, since $h\geq0$ on $S \times (0,\tilde{T}_0)$, there holds
\begin{equation}\label{eqn:basic-ineq}
\int_0^{\widetilde{T}_0 }\int_{\partial\Omega} (\partial_{\nu_A} u_1 - \partial_{\nu_A} u_2) h(x,t) \d \sigma(x) \d t
\leq \sup_{S\times (0,T)} |\partial_{\nu_A} u_1 - \partial_{\nu_A} u_2|\int_0^{\widetilde{T}_0}\int_S  h(x,t) \d \sigma(x) \d t  .
\end{equation}
Let 
\begin{equation*}
    \mu :=\sup_{S\times (0,T)} |\partial_{\nu_A} u_1 - \partial_{\nu_A} u_2|.
\end{equation*}
Combining the estimate \eqref{eqn2} and \eqref{eqn:basic-ineq} with the identity \eqref{main equality} gives
\begin{equation}
\label{stab. est main ineq.} 
(F(r_0) - \sqrt{2}MC_1\varepsilon) \int_{\mathcal{T}_{0,\varepsilon}} \int_{\mathcal{O}_{x',\varepsilon}} v(x,t) \d x\d t
\leq \mu \int_0^{\widetilde{T}_0}\int_S  h(x,t) \d \sigma(x) \d t  .
\end{equation}
From now on, we analyze the two cases depending on whether there exists a point $x' \in S$ such that $g(x') = R$. 

\medskip\noindent\textbf{Case 1: There exists $x' \in S$ such that $g(x') = R$.}  First, choose $x'\in S$ such that 
$$
\text{dist}(x',\partial S) = d^*= \min_{\substack{x \in S, g(x) = R}} \text{dist}(x,\partial S),
$$
and additionally assume $\varepsilon < d^*$ so that $\mathcal{O}_{x',\varepsilon}\cap S = \mathcal{O}_{x
',\varepsilon} \cap \partial\Omega$. Next, fix $h\in \mathcal{C}^\infty(\partial\Omega \times [0,\widetilde{T}_0])$ to be a function supported only on $(\mathcal{O}_{x', d^*}\cap S) \times [0,\widetilde{T}_0]$ satisfying the compatibility condition $h(x,\widetilde{T}_0) \equiv 0$ with $h(x',t_0) = 1$. Moreover, fix $\|h\|_{\mathcal{C}^{2+\alpha,1+\frac\alpha 2}(\partial\Omega \times [0,\widetilde{T}_0])} = \tilde{C} = \tilde{C}(m_0,d^*)>0$. Now, fix $\delta>0$ to be arbitrary small. For any $(x,t)$ satisfying $|x-x'|<\delta$ and $|t-t_0|<\delta$, applying the mean value theorem and Lemma \ref{lem:schauder-type est} (ii), we obtain
\begin{align*}
|v(x,t) -v(x',t_0) |
&\leq  \|v\|_{C^{1,1}(\overline{\Omega}\times [0,\widetilde{T}_0])} (|x-x'|+|t-t_0|) \\
&\leq 2C_2\delta\|h\|_{C^{2+\alpha,1+\frac\alpha2}(\partial\Omega\times [0,\widetilde{T}_0])} 
\leq 2\tilde{C}C_2
\delta.
\end{align*}
Thus, it follows from the triangle inequality that 
$$
\begin{aligned}
v(x,t) \geq v(x',t_0) - |v(x',t_0) - v(x,t)| \geq 1 - 2\tilde{C}C_2\delta.
\end{aligned}
$$
Hence, by choosing $c < \min\{1,(4\tilde{C}C_2)^{-1}\}$ and taking $\delta = c \varepsilon$ we have $v(x,t) \geq 1/2$ on $\mathcal{O}_{x',\delta} \times \mathcal{T}_{0,\delta}$ and also 
$\mathcal{O}_{x',\delta}\times \mathcal{T}_{0,\delta} \subset \mathcal{O}_{x',\varepsilon}\times \mathcal{T}_{0,\varepsilon}$. Thus we arrive at the following lower bound
$$
\int_{\mathcal{T}_{0,\varepsilon}} \int_{\mathcal{O}_{x',\varepsilon}} v(x,t) \d x\d t \geq \int_{\mathcal{T}_{0,\delta}} \int_{\mathcal{O}_{x',\delta}} v(x,t) \d x \d t \geq \frac 12 \int_{\mathcal{T}_{0,c\epsilon}} \int_{\mathcal{O}_{x',c\epsilon}} 1 \d x \d t  = \frac {C_\Omega c^{d+1}} {2} \varepsilon^{d+1} ,
$$
where the constant $C_\Omega>0$  depends only on $\Omega$. 
This estimate, the inequality \eqref{stab. est main ineq.} and the estimate $\|h\|_{L^1(S\times (0,\widetilde{T}))} \leq \tilde{C}_{\Omega} \tilde{C}$ imply
\begin{equation}
    F(r_0)\leq \frac{2\tilde{C}_\Omega \tilde{C}}{C_\Omega c^{d+1}}\varepsilon^{-d-1} \mu +\sqrt{2}MC_1 \varepsilon.
    \label{eqn:F(r_0)<=eps^-1+eps}
\end{equation}
We claim that the estimate \eqref{eqn:F(r_0)<=eps^-1+eps} implies 
\begin{equation}
\label{step 1 case 1}
F(r_0) \leq C\mu^{\frac{1}{d+2}}, 
\end{equation}
with $C = C(\Omega,T,A,{B},M,\kappa,d^*)>0$. If $\mu = 0$, we obtain \eqref{step 1 case 1} directly by sending $\varepsilon$ to 0 in \eqref{eqn:F(r_0)<=eps^-1+eps}. If $\mu^{\frac{1}{d+2}} < \min\{m_0,d^*\} $, 
then we may choose $\varepsilon = \mu^{\frac{1}{d+2}}$, and the estimate \eqref{eqn:F(r_0)<=eps^-1+eps} yields $F(r_0) \leq C \mu^{\frac{1}{d+2}}$, where  all the constants in \eqref{eqn:F(r_0)<=eps^-1+eps} are collected into $C$. Last, suppose $\mu^{\frac{1}{d+2}} \geq \min\{m_0,d^*\}$. In the case $\mu^{\frac{1}{d+2}} \geq m_0$, we choose $\varepsilon = \tilde{c}m_0$, where $\tilde{c}=\tilde{c}(d^*)>0$ is such that $\tilde{c}m_0 \leq \min\{m_0,d^*\}$. Then, from \eqref{eqn:F(r_0)<=eps^-1+eps}, we have 
\begin{equation}
\label{F(r_0) <= Cmu case}
F(r_0) \leq C\left( \frac{\mu}{\tilde{c}^{d+1}m_0^{d+1}} + \tilde{c}m_0\right) \implies m_0^{d+1} F(r_0) \leq C(\mu + m_0^{d+2}) \leq C \mu. 
\end{equation}
Meanwhile,  by the mean value theorem,
$$
t_0 = f^{-1}(r_0R^{-1}) - f^{-1}(s_0R^{-1}) \geq R^{-1}(r_0-s_0)\min_{t \in (T_1,T_2)} (f^{-1})'(t).
$$
Then, by the mean value theorem again and the condition $F(s_0)=0$ from \eqref{asmp. on F}, we have
$$
\frac{F(r_0)}{r_0-s_0}=\frac{F(r_0)-F(s_0)}{r_0-s_0}\leq  \|F\|_{\mathcal{C}^1([0,R])} \leq M.
$$
Combining the last two estimates as well as using the elementary identity $(f^{-1})'(y) = (f'(f^{-1}(y)))^{-1}$, we have
$$
\frac{F(r_0)}{t_0} \leq \frac{MR}{\min_{t \in (T_1,T_2)} (f^{-1})'(t)} \leq MR\|f\|_{\mathcal{C}^1([T_1,T_2])} \leq MR \kappa.
$$
By similarly considering $\frac{F(r_0)}{\widetilde{T}_0-t_0}$ and then raising the inequality to the power of $d+1$, we deduce
$$
\frac{[F(r_0)]^{d+1}}{m^{d+1}_0} \leq (MR \kappa)^{d+1}.
$$
Multiplying this inequality with \eqref{F(r_0) <= Cmu case} yields
$$
[F(r_0)]^{d+2} \leq C\mu,
$$
which directly implies the claim \eqref{step 1 case 1}. Lastly, in the case $\mu^{\frac{1}{d+2}} < m_0$ but $\mu^{\frac{1}{d+2}} \geq d^*$, we have 
$$
F(r_0) \leq M \frac{\mu^{\frac{1}{d+2}}}{d^*} \leq C\mu^{\frac{1}{d+2}}. 
$$
Hence the claim \eqref{step 1 case 1} holds in all cases. Next we prove the estimate \eqref{stab. est. for x' in S} by iteration. Now fix $\widetilde{T} = \widetilde{T}_1$. Since $s_0\leq u_2\leq s_2$ in the region $\Omega \times (0,\widetilde{T}_1)$ by the maximum principle, using the identity
 \eqref{main equality} and the triangle inequality, we get
\begin{align*} 
\left\vert\int_{\{s_1\leq u_2\leq s_2\}} F(u_2(x,t)) v(x,t) \d x \d t\right\vert \leq & \left\vert \int_0^{\widetilde{T}_1} \int_{S} (\partial_{\nu_A} u_1-\partial_{\nu_A} u_2) h \d\sigma(x) \d t \right\vert \\
 &+ \left\vert\int_{\{s_0\leq u_2\leq s_1\}} F(u_2(x,t)) v(x,t) \d x \d t  \right\vert.
\end{align*}
Next, note that $F(s) \leq 0$ for  $s_1 \leq s \leq s_2$ by the assumption that $F$ changes sign at $s_1$ and $F(s)\geq 0$ for $s_0\leq s\leq s_1$. Thus, by the positivity of $v$, we have
$$
\int_{\{s_1\leq u_2\leq s_2\}} |F(u_2(x,t))| v(x,t) \d x \d t = \left\vert \int_{\{s_1\leq u_2\leq s_2\}} F(u_2(x,t)) v(x,t) \d x \d t\right\vert ,
$$
and moreover, since $F(s) \leq F(r_0)$ for $s_0\leq s \leq s_1$, from the estimate \eqref{step 1 case 1}, we deduce 
\begin{align*}
\int_{\{s_1\leq u_2\leq s_2\}} |F(u_2(x,t))| v(x,t) \d x \d t
\leq& \mu \int_0^{\widetilde{T}_1} \int_{S}  h \d\sigma(x) \d t +  C\mu^{\frac{1}{d+2}}\int_{\{s_0\leq u_2\leq s_1\}} v(x,t) \d x \d t.
\end{align*}
Fix $m_1 = \min \{t_1-\widetilde{T}_0,\widetilde{T}_1 - t_1\}$ and fix $0< \epsilon < \min\{m_1,d^*\}$. We also set $\mathcal{T}_{1,\varepsilon} = [t_1-\varepsilon,t_1+\varepsilon]$. We choose $h \in \mathcal{C}^\infty(\partial\Omega \times [0,\widetilde{T}_1])$ to be a  function supported only on $(\mathcal{O}_{x', d^*}\cap S) \times [0,\widetilde{T}_1]$ satisfying the compatibility condition $h(x,\widetilde{T}_1) \equiv 0$ and also satisfying $h(x',t_1) = 1$. Similarly, suppose that $\|h\|_{\mathcal{C}^{2+\alpha,1+\frac\alpha2}(\partial\Omega\times [0,\widetilde{T}_1])} = \tilde{C} = \tilde{C}(m_1,d^*) > 0$. Then, repeating the preceding argumentation yields 
$$
\int_{\{s_1\leq u_2\leq s_2\}}|F(u_2(x,t))|v(x,t) \d x\d t\geq \int_{\mathcal{T}_{1,\varepsilon}} \int_{\mathcal{O}_{x',\varepsilon}}|F(u_2(x,t))|v(x,t) \d x\d t\geq C(|F(r_1)| - \varepsilon)\varepsilon^{d+1}.
$$
with $C = C(\Omega,T,A,{B},M,\kappa,m_1,d^*)>0$. Meanwhile, by Lemma \ref{lem:schauder-type est}(iii), we have 
 $$
    \int_{\{s_0 \leq u_2 \leq s_1 \}}v(x,t) \d x\d t \leq C_3\|h\|_{L^1(\partial\Omega \times (0,\widetilde{T}))} \leq C_3\tilde{C}_\Omega\tilde{C}.
$$ 
Hence, we derive  
$$
|F(r_1)|\leq C \varepsilon^{-d-1}\mu ^{\frac {1}{d+2}} + \varepsilon.
$$
Hence, by choosing $\varepsilon$ appropriately (as in the justification of \eqref{step 1 case 1}), we obtain 
$$
|F(r_1)| \leq C\mu^{\frac{1}{(d+2)^2}} ,
$$
with $C = C(\Omega,T,A,{B},M,\kappa,d^*)>0$. By iterating this process, we obtain the estimate \eqref{stab. est. for x' in S} in Case 1. 

\noindent\textbf{Case 2: There does not exist $x' \in S$ such that $g(x') = R$.}  
Let $v_M$ be defined as in Lemma \ref{lem: v_M leq v leq v_0}. Now, let $S' \subset S$ be the largest open connected subset of $S$. By further assuming that $\varepsilon < |S'|/4$, we can find $\tilde{S} \subset S'$ such that $\tilde{S} \cap \mathcal{O}_{x',\varepsilon} = \emptyset$. Next, let $h \in \mathcal{C}^\infty(\partial\Omega \times [0,\widetilde{T}_0])$ with $h \not\equiv 0$ such that $h$ is supported only on $\tilde{S} \times [0,\widetilde{T}_0]$ and satisfies the compatibility conditions $h(x,\widetilde{T}_0) = 0$. Let $v_M$ be defined in Lemma \ref{lem: v_M leq v leq v_0}. Now, by Hopf's lemma (see e.g. \cite[Proposition 1.48]{choulli2009introduction}), we have
$$
\frac{\partial v_M}{\partial\nu} (x,t)<0, \quad (x,t) \in  (\partial\Omega \setminus \tilde{S})\times (0,\widetilde{T}_0).
$$ 
Using boundary normal coordinates, we can conclude 
\begin{equation}
    \label{eqn3}
    v_M(x,t) \geq Cd(x,\partial\Omega), \quad (x,t) \in \mathcal{O}_{x',\varepsilon}\times \mathcal{T}_{0,\varepsilon},
\end{equation}
with $C = C(\Omega,T,A,{B},M)>0$. By Lemma \ref{lem: v_M leq v leq v_0} and the estimate \eqref{eqn3}, we find
$$
\int_{\mathcal{T}_{0,\varepsilon}} \int_{\mathcal{O}_{x',\varepsilon}} v(x,t) \d x \d t \geq \int_{\mathcal{T}_{0,\varepsilon}}\int_{\mathcal{O}_{x',\varepsilon}} v_M(x,t) \d x \d t \geq C\int_{\mathcal{T}_{0,\varepsilon}} \int_{\mathcal{O}_{x',\varepsilon}}  d(x,\partial\Omega) \d x\d t \geq C\varepsilon^{d+2}.
$$
This estimate, \eqref{stab. est main ineq.} and the inequality  $\|h\|_{L^1(S\times (0,\widetilde{T}))}\leq C$ imply
\begin{equation}
    F(r_0) \leq C \left( \varepsilon^{-d-2} \mu + \varepsilon \right).\label{eqn:F(r_0_<=eps^-2+eps}
\end{equation}
Similar to  Case 1, the estimate \eqref{eqn:F(r_0_<=eps^-2+eps} implies
$$
F(r_0) \leq C\mu ^{\frac {1}{d+3}}.
$$
Then, the desired result \eqref{stab. est. for x' not in S} follows by the iteration argument. 
\end{proof}

Last, we prove Corollary \ref{cor:improved-stability}.

\begin{proof}[Proof of Corollary \ref{cor:improved-stability}]
Let $s,s'\in[0,R]$ with $s\neq s'$. Since $F\in \mathcal{C}^2([0,R])$, by the Taylor expansion with a Lagrangian error term, we have
\begin{equation}
\label{taylon exp.}
F(s) = F(s') + F'(s')(s-s') + \tfrac 12 F''(\xi) (s-s')^2,\quad \xi\in(\min(s,s'),\max(s,s') ) .
\end{equation}
 We define $s_0$, $r_0$, $x'$, $t_0$, $\mathcal{O}_{x',\varepsilon}$ and $\mathcal{T}_{0,\varepsilon}$ as in the proof of Theorem \ref{thm:stability}. Then, by taking $s' = u_2(x',t_0) = r_j$, we have by its construction $F'(r_j) = 0$, and we can derive for $(x,t) \in \mathcal{O}_{x',\varepsilon} \times \mathcal{T}_{0,\varepsilon}$ that 
\[
\begin{aligned}
F(u_2(x,t)) &\geq F(u_2(x', t_0)) - \left| F(u_2(x', t_0)) - F(u_2(x,t)) \right| \\
& \geq F(u_2(x',t_0)) - \tfrac 12 \|F\|_{\mathcal{C}^2([0,R])} |u_2(x',t_0)-u_2(x,t)|^2\\
&\geq F(\varphi(t_0) g(x')) - \tfrac 1 2
\|F\|_{\mathcal{C}^2([0,R])} \|u_2\|^2_{\mathcal{C}^1(\overline{\Omega} \times [0,T])} |(x', t_0) - (x,t)|^2 \\
&\geq F(r_0) - C^2_1M \varepsilon^2,
\end{aligned}
\]
which is an analogue of the estimate \eqref{eqn: F>=F(r_0) - C*eps}. Now we split the analysis into two cases like before. In Case 1, repeating the preceding argument, we arrive at
\begin{equation}
\label{eqn4}
    F(r_0)\leq C\left(\varepsilon^{-d-1} \mu + \varepsilon^2\right),\quad \mbox{with }\mu = \sup_{S\times (0,T)} |\partial_{\nu_A} u_1 - \partial_{\nu_A} u_2 |,
\end{equation}
with $\varepsilon < \min(m_0,d^*)$. We claim that the  estimate \eqref{eqn4} implies
\begin{equation}
\label{eqn4.5}
  F(r_0)\leq C\mu^{\frac {2}{d+3}}.
\end{equation}
The cases $\mu = 0$ and $\mu ^{\frac{1}{d+3}} < \min \{m_0,d^*\}$ can be treated similarly as in the proof of Theorem \ref{thm:stability}. In the case $\mu^{\frac{1}{d+3}} \geq m_0$, by choosing $\varepsilon = \tilde{c}m_0$ and repeating the derivation of \eqref{F(r_0) <= Cmu case}, we obtain
\begin{equation}
\label{eqn5}
m_0^{d+1}F(r_0) \leq C\mu.
\end{equation}
Meanwhile, by the mean value theorem,
$$
t_0 = f^{-1}(r_0R^{-1}) - f^{-1}(s_0R^{-1}) \geq R^{-1}(r_0-s_0)\min_{t \in (T_1,T_2)} (f^{-1})'(t).
$$
Then, by the Taylor expansion and noting $F(s_0)=0$, we have
$$
\frac{2F(r_0)}{(r_0-s_0)^2}=\frac{2(F(r_0)-F(s_0))}{(r_0-s_0)^2}\leq  \|F\|_{\mathcal{C}^2(\R)} \leq M.
$$
Combining the preceding two estimates yields
$$
\frac{F(r_0)}{t^2_0} \leq \frac M2\left(\frac{R}{\min_{t \in (T_1,T_2)} (f^{-1})'(t)} \right)^2\leq \frac {MR^2}{2}\|f\|^2_{\mathcal{C}^1([T_1,T_2])} \leq MR^2 \kappa^2.
$$
By similarly treating $\frac{F(r_0)}{\widetilde{T}_0-t_0}$, we can conclude 
$$
\frac{[F(r_0)]^\frac{d+1}{2}}{m^{d+1}_0} \leq (MR^2 \kappa^2)^{\frac{d+1}{2}}.
$$
Then using the estimate \eqref{eqn5}, we derive
$$
F(r_0)^{\frac{d+1}{2}+1} \leq C\mu,
$$
which directly implies the desired claim \eqref{eqn4.5}. Finally, the case  $\mu^{\frac{1}{d+3}}<m_0$ but $\mu^{\frac{1}{d+3}} \geq d^*$ can be treated similarly as in the proof of Theorem \ref{thm:stability}. Repeating the preceding iteration argument similarly yields \eqref{x in S, increased regularity on F}. The same augmentation can be applied to the proof of Case 2 and thus we obtain the estimate \eqref{x not in S, increased regularity on F}. This concludes the proof of the corollary. 
\end{proof}

\section{Numerical experiments and discussions}
\label{sec:numer}
In this section, we present several numerical experiments to illustrate the feasibility of numerical reconstruction, to complement the theoretical results in Section \ref{sec:main}. In the numerical experiments, we fix $T=1$, and let $\Omega \subset \mathbb{R}^2$ be the disk with a radius $0.5$ centred at the origin. Consider the following semilinear parabolic model:
$$
\left\{\begin{aligned}
\partial_t u - \Delta u + F(u) &= 0, && \text{in } \Omega \times (0,T), \\
u &= h, & & \text{on } \partial\Omega \times (0,T), \\
u(\cdot, 0) &= 0, && \text{in } \Omega.
\end{aligned}\right.
$$
The inverse problem is to recover the semilinear term $F(u)$ from partial boundary measurements, i.e., the flux $m(x,t) = \partial_\nu u(x,t)$ on a subset $S \subseteq \partial\Omega$, recorded over the time interval $(0,T)$.

\subsection{Numerical algorithm}

For the numerical reconstruction, the standard regularized output least-squares approach \cite{ItoJin:2015} reads:
$$
J_\lambda(F) = \underbrace{\frac12  \int_0^T \int_S |\partial_\nu u_F(x,t) - m(x,t)|^2 \, \mathrm{d}\sigma(x) \mathrm{d}t}_{J(F)} +  \frac{\lambda}{2}  \|F(u_F)\|_{L^2(0,T;L^2(\Omega))}^2,
$$
where $\lambda > 0$ is the regularization parameter, and $u_F$ denotes the state corresponding to the semilinear term $F(u)$. Note that the penalty term $\|F(u_F)\|_{L^2(0,T;L^2(\Omega))}^2$ views $F(u_F)$ as an unknown source, following the inverse source problem methodology in the stability analysis, and it does not penalize $F$ explicitly. To minimize the functional $J_\lambda$, we can employ a gradient-based iterative scheme. The Fr\'{e}chet derivative $D_FJ_\lambda(F)$ of the functional $J_\lambda(F)$ in any direction $G$ can be efficiently computed using an adjoint state $v$, which satisfies the following backward problem:
\begin{equation}
 \left\{   \begin{aligned}
-\partial_t v - \Delta v  &= 0, && \text{in } \Omega \times (0,T), \\
v(\cdot, T) &= 0, && \text{in } \Omega, \\
  v &= \chi_S (\partial_\nu u_F - m),& & \text{on } \partial\Omega \times (0,T),
\end{aligned}\right.\label{eqn:backw}
\end{equation}
where $\chi_S$ denotes the characteristic function of the set $S$. Let $w=D_F u_F\cdot G$ be the linearization (directional derivative) of $u_F$ in $F$ along the direction $G$, which satisfies
\begin{equation}\label{eqn:lin-problem} \left\{   \begin{aligned}
\partial_t w - \Delta w+F'(u_F)w+G(u_F)  &= 0, && \text{in } \Omega \times (0,T), \\
w(\cdot, 0) &= 0, && \text{in } \Omega, \\
  w &= 0,& & \text{on } \partial\Omega \times (0,T).
\end{aligned}\right.
\end{equation}
The notation $\langle\cdot,\cdot\rangle_{L^2(0,T;L^2(\Omega))}$ denotes the $L^2(0,T;L^2(\Omega))$ inner product, and likewise $\langle\cdot,\cdot\rangle_{L^2(0,T;L^2(S))}$.
By integration by parts, we obtain the directional derivative $D_FJ(F)G$:
$$
\begin{aligned}
&D_F J(F)G =  \langle v, \partial_\nu w \rangle_{L^2(0,T;L^2(S))}
=  \langle v, \Delta w \rangle_{L^2(0,T;L^2(\Omega))} +  \langle \nabla v, \nabla w \rangle_{L^2(0,T;L^2(\Omega))} \\
=&  \langle v, \partial_t w + G(u_F) + F'(u_F) w \rangle_{L^2(0,T;L^2(\Omega))} -  \langle \Delta v, w \rangle_{L^2(0,T;L^2(\Omega))}
=  \langle v, G(u_F) + F'(u_F) w \rangle_{L^2(0,T;L^2(\Omega))}.
\end{aligned}
$$
Next, consider the functional
$$
d(F,G) := \tfrac12 \| F(u_F) - G(u_G) \|_{L^2(0,T;L^2(\Omega))}^2,
$$
whose Fréchet derivative at $G=0$ satisfies the relation
$$
D_F d(\cdot, 0)(G) =  \langle F(u_F), G(u_F) + F'(u_F) w \rangle_{L^2(0,T;L^2(\Omega))}.
$$
Thus, the derivative of $J_\lambda$ takes the form
$$
D_F J_\lambda(F)G = \langle v + \lambda F(u_F), G(u_F) + F'(u_F) w \rangle_{L^2(0,T;L^2(\Omega))}.
$$
At the minimizer $F$ of $J_\lambda$, we know that for every $G$, we must have 
$$ \langle v + \lambda F(u_F), G(u_F) + F'(u_F) w \rangle_{L^2(0,T;L^2(\Omega))} = 0.$$ That is $$ \langle v + \lambda F(u_F), G(u_F) + F'(u_F) (D_F u_F\cdot G) \rangle_{L^2(0,T;L^2(\Omega))} = 0.$$
Hence, the first-order necessary condition reads 
$$v+\lambda F(u_F)\perp W_F:= \{G(u_F) + F'(u_F)(D_F u_F\cdot G):G\in C_c^\infty [0,R]\},$$
and we may develop a n iterative reconstruction algorithm based on the relation $v+\lambda F(u_F)=0$. This relation relates the adjoint source function $-v/\lambda$ in the spatial-time domain $\Omega\times (0,T)$ with the semilinear function $F$ on $[0,R]$. This observation directly motivates the algorithm using the piecewise regression \eqref{eqn:sub-cand} and the fixed point iteration \eqref{eqn:ITERATIONS}.

% We adopt a gradient-descent update, and at each step, structural constraints, e.g., $F(0)=0$ and $F'(s)\geq 0$, are enforced through projection. The numerical results are presented in Figs.~\ref{Fig:Method2_1} and~\ref{Fig:Method2_2}.

% \begin{algorithm}[htbp]
% \caption{Iterative reconstruction of the semilinear term $F$ by functional targeting {\color{blue} with one measurement}}
% \label{alg:reconstruction}
% \begin{algorithmic}[1]
% \STATE \textbf{Initialize:} Choose an arbitrary initial guess $f_0$ and a constant $M > 0$. Set iteration counter $k = 0$.
% \REPEAT
%     \STATE Solve the forward parabolic PDE to obtain $u_{f_k}$ using the current guess $f_k$.
%     \STATE Evaluate the residual on the boundary and determine the adjoint boundary condition: $g_{\text{adj}} = \chi_S(x) \big( \partial_\nu u_{f_k} - m \big)$.
%     \STATE Solve the backward adjoint problem \eqref{eqn:backw} to obtain the adjoint state $v_k$.
%     \STATE Define the intermediate function by solving the minimization problem: $f_{v_k} := \arg\min_{f} \|f(u_{f_k}) - v_k\|_2$.
%     \STATE Update the semilinear term using the relaxation formula:
%     $f_{k+1} := \frac{-f_{v_k}}{M+\lambda} + \frac{M f_k}{M+\lambda}$.
%     \STATE $k \leftarrow k + 1$
% \UNTIL{a stopping criterion is met (e.g., $\|f_k - f_{k-1}\| < \epsilon$).}
% \end{algorithmic}
% \end{algorithm}

Specifically, we propose a novel iterative reconstruction algorithm with multiple measurements in Algorithm \ref{alg:reconstruction_multi}. The algorithm consists of two loops: one inner loop going over all the flux measurements $\{m_i\}_{i=1}^I$, and for each measurement $m_i$, we construct one candidate $F_{v_k}^{(i)}$ via the standard least-squares fitting over a suitable subdomain $\Omega_+\subseteq \Omega$ (to be determined) in \eqref{eqn:sub-cand}. This step provides a map from the source function (the adjoint state) to the nonlinearity $F$ on the interval $[0,R]$, following the methodologies for the inverse source problem. The estimation of the confidence weights $\{w_i(u)\}_{i=1}^I$ can be achieved by the standard kernel density estimation \cite{KDE}. The aggregation step combines the individual estimators $\{F_{v_k}^{(i)}\}_{i=1}^I$, using the confidence weights $\{w_i(u)\}_{i=1}^I$ to balance different candidates. The convergence test at Step 14 can be based on the condition $\|F_k - F_{k-1}\|_{L^2(0,R)} \leq \epsilon$, for a prescribed tolerance $\epsilon$ or simply reaching the maximum number of iterations. In the numerical experiments, the algorithm converges rapidly, often in about five iterations.

\begin{algorithm}[htb!]
\caption{Iterative reconstruction of the semilinear term $F(u)$}
\label{alg:reconstruction_multi}
\begin{algorithmic}[1]
\STATE \textbf{Input}: The boundary flux measurements $\{m_i\}_{i=1}^I$. 
\STATE \textbf{Initialize:} Set the initial guess $F_0$, constants $M > 0$, $\lambda > 0$, $k = 0$.
\REPEAT
    \FOR{$i = 1$ \TO $I$}
        \STATE Solve problem \eqref{IBVP} using $F_k$ for the state $u_{F_k}^{(i)}$ (corresponding to the excitation for $m_i$).
        \STATE Compute the adjoint boundary condition:
        $g_{\text{a}}^{(i)} = \chi_S\big( \partial_\nu u_{F_k}^{(i)} - m_i \big)$.
        \STATE Solve the adjoint state $v_k^{(i)} $ from problem \eqref{eqn:backw} with $g=g_\text{a}^{(i)}$.
        \STATE Determine the one-measurement candidate  \begin{equation}\label{eqn:sub-cand} F_{v_k}^{(i)} := \arg\min_{f} \big\| f(u_{F_k}^{(i)}) - v_k^{(i)} \big\|_{L^2(\Omega_+)}.
        \end{equation}
    \ENDFOR
    \STATE Estimate the confidence weight $w_i(u)=p_i(u)$ of $u_{F_k}^{(i)}$, $i=1,\ldots,I$, on $\Omega_+$ for $u\in[0,R]$. 
    \STATE Aggregate the candidates
        \begin{equation}
            F_{v_k}(u) = \frac{\sum_{i=1}^{I} w_i(u) \, F_{v_k}^{(i)}(u)}{\sum_{i=1}^I w_i(u)}.
        \end{equation}
    (If all weights vanish at a given $u$, set $F_{v_k}(u) = F_k(u)$.)
    \STATE Update the semilinear term $F(u)$ using relaxation:
        \begin{equation}
            F_{k+1} := \dfrac{-F_{v_k}}{M+\lambda} + \dfrac{M\,F_k}{M+\lambda}.\label{eqn:ITERATIONS}
        \end{equation}
    \STATE $k \leftarrow k + 1$
\UNTIL{The convergence criterion is met and return $F=F_k$.} 
\end{algorithmic}
\end{algorithm}

\subsection{Numerical results and discussions}
%\begin{figure}[hbtp!]
%    \centering
%    \includegraphics[width=0.5\linewidth]{Figures/G.png}
%    \caption{The plot of $g_0(\theta,0.7)$}
%    \label{fig:G}
%\end{figure}
In the experiments, we fix $R=1$, and take four measurements in order to improve the stability of the reconstruction process: $(u_0,h)=(0,h_0)$, $(u_0,h)=(1,1-h_0)$, $(u_0,h)=(0.75,0.75-0.75h_0)$, $(u_0,h)=(0.25,0.25+0.75h_0)$, with $h_0(\theta,t) = (3 \cos^2 (\frac\theta2)\exp(-0.5\theta^2)-0.9)\min(0.7,t)$, for $\theta\in [-\pi,\pi)$. We choose the true nonlinearity $F^\dagger$ by $F^\dagger(s)=s^3$. In Algorithm \ref{alg:reconstruction_multi}, we take $\Omega_+:=\{x\in\Omega: x_1>0\}$, which is close to the region where the boundary flux focuses (i.e., with informative data). We employ the Galerkin finite element method (FEM) for the spatial discretization, and the implicit Euler scheme for the time discretization \cite{Thomee:2006}. 
To illustrate the stability of the inverse problem, we use also noisy data, generated by adding noise to the exact boundary flux $\partial_\nu u$:
\[
m^\varepsilon(x,t) = \partial_\nu u(x,t) (1+ \varepsilon \cdot \xi(x,t) ),\quad (x,t)\in S\times(0,T),
\]
where $\varepsilon \in \{0, 0.01, 0.1, 0.2\}$ denotes the relative noise level, and for each $(x,t)$, $\xi(x,t)$ follows the standard Gaussian distribution independently. To examine the impact of the measurement boundary $S$ on the reconstruction, we test six distinct configurations of $S$; see  Fig. \ref{fig:locations} for a schematic illustration. We denote the case with the boundary portion $S$ not containing the point $(1,0)$ as `far' data, since by construction, the flux focuses its density near $\theta = 0$. In the experiments, we set $M=0.2$ and $\lambda=10^{-5}$ in Algorithm \ref{alg:reconstruction_multi}.
The variations of the loss (for each flux measurement) and the $L^2(0,R)$ error of the reconstruction are shown in Fig. \ref{fig:min}. Note that we do not assume the knowledge of $F$ outside $[0,R]$, and also perform optimization outside $[0,R]$ but do not show the relevant results. Numerically, Algorithm \ref{alg:reconstruction_multi} converges rather steadily in terms of both objective value and $L^2(0,R)$ error, with the convergence reached in a few iterations.

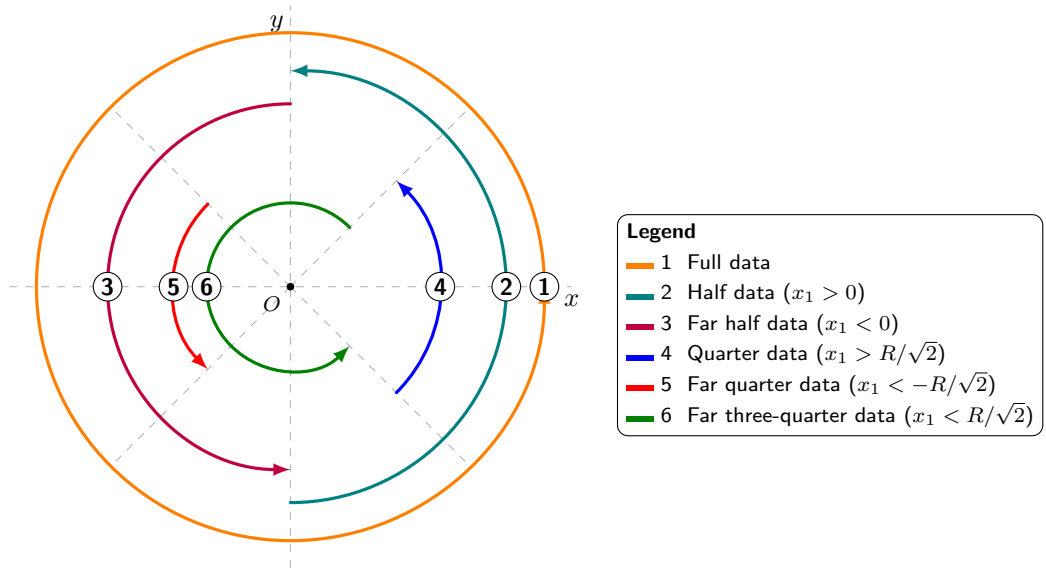
\begin{figure}[htb!]
  \centering
  \begin{tikzpicture}[
    scale=1.2,
    >=latex,
    arc style/.style={very thick, ->, line cap=round},
    mark line/.style={thin, dashed, gray!60},
    % 带圆圈的数字样式
    circ node/.style={circle, draw=black, fill=white, inner sep=1pt, font=\small\bfseries\sffamily},
  ]
  \def\R{2.8}

  % 坐标系与45°分界线
  \draw[mark line] (-\R-0.3,0) -- (\R+0.3,0);
  \draw[mark line] (0,-\R-0.3) -- (0,\R+0.3);
  \draw[mark line] (0,0) -- (45:\R);
  \draw[mark line] (0,0) -- (-45:\R);
  \draw[mark line] (0,0) -- (135:\R);
  \draw[mark line] (0,0) -- (225:\R);
  \fill (0,0) circle (1.2pt) node[below left] {\footnotesize$O$};
  \node at (\R+0.3, -0.15) {$x$};
  \node at (-0.15, \R+0.1) {$y$};

  % 各层半径
  \def\rA{\R}
  \def\rB{0.85*\R}
  \def\rC{0.72*\R}
  \def\rD{0.59*\R}
  \def\rE{0.46*\R}
  \def\rF{0.33*\R}

  % ---- 六个同心圆弧，每个旁边放一个数字 ----
  % 1. Full data（整圆），序号放在右侧中部
  \draw[arc style, orange] (0:{\rA}) arc (0:360:{\rA});
  \node[circ node] at (0:{\rA}) {1};

  % 2. Half data（右半圆）
  \draw[arc style, teal] (-90:{\rB}) arc (-90:90:{\rB});
  \node[circ node] at (0:{\rB}) {2};

  % 3. Far half data（左半圆）
  \draw[arc style, purple] (90:{\rC}) arc (90:270:{\rC});
  \node[circ node] at (180:{\rC}) {3};

  % 4. Quarter data（右侧扇形）
  \draw[arc style, blue] (-45:{\rD}) arc (-45:45:{\rD});
  \node[circ node] at (0:{\rD}) {4};

  % 5. Far quarter data（左侧扇形）
  \draw[arc style, red] (135:{\rE}) arc (135:225:{\rE});
  \node[circ node] at (180:{\rE}) {5};

  % 6. Far three-quarter data（除右侧扇形外）
  \draw[arc style, green!50!black] (45:{\rF}) arc (45:315:{\rF});
  \node[circ node] at (180:{\rF}) {6};

  % 图例
  \node[draw, fill=white, rounded corners, anchor=north east, 
        font=\footnotesize\sffamily, align=left] 
        at (\R+5.5, \R-2) {
    \textbf{Legend}\\[2pt]
    \textcolor{orange}{\rule{1.2em}{2pt}} 1~~Full data\\[2pt]
    \textcolor{teal}{\rule{1.2em}{2pt}} 2~~Half data ($x_1>0$)\\[2pt]
    \textcolor{purple}{\rule{1.2em}{2pt}} 3~~Far half data ($x_1<0$)\\[2pt]
    \textcolor{blue}{\rule{1.2em}{2pt}} 4~~Quarter data ($x_1>R/\sqrt2$)\\[2pt]
    \textcolor{red}{\rule{1.2em}{2pt}} 5~~Far quarter data ($x_1<-R/\sqrt2$)\\[2pt]
    \textcolor{green!50!black}{\rule{1.2em}{2pt}} 6~~Far three-quarter data ($x_1<R/\sqrt2$)
  };

  \end{tikzpicture}
  \caption{Schematic illustration of the measurement configuration $S$.}
  \label{fig:locations}
\end{figure}

\begin{figure}[hbt!]
    \centering
    \begin{tabular}{cc}
    \includegraphics[width=0.5\linewidth]{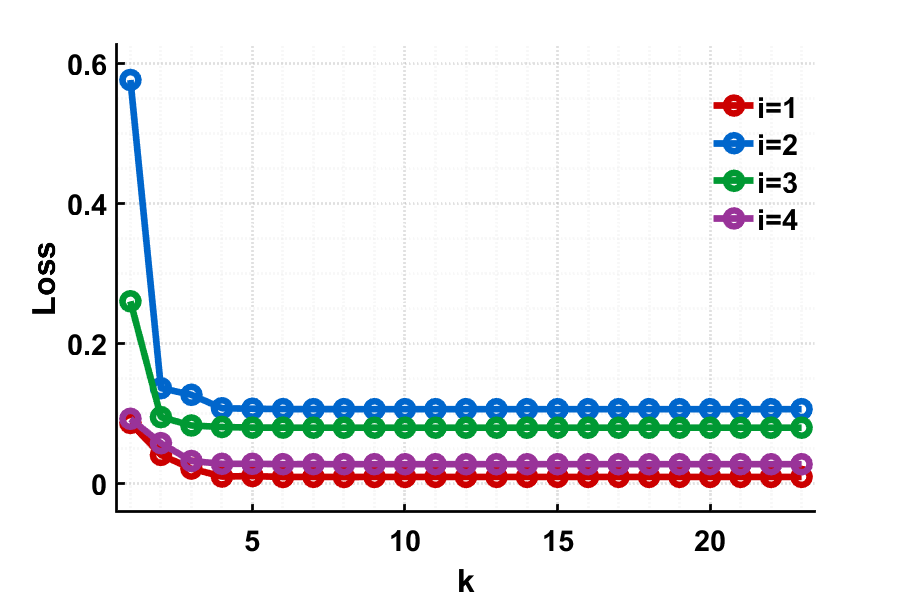}
         &    \includegraphics[width=0.5\linewidth]{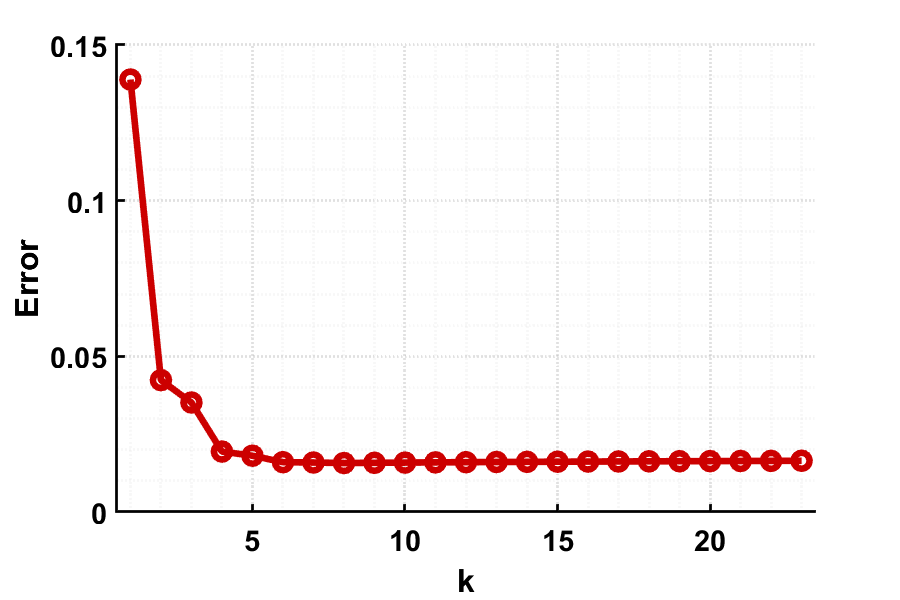}
    \end{tabular}
    \caption{The variation of the individual losses $J_\lambda^i(F)$ for the flux measurements $\{m_i\}_{i=1}^4$ during the iteration (left); and the $L^2(0,R)$ error of the recovered semilinear term $F$ during the iteration (right), for the case of full data without noise.}
    \label{fig:min}
\end{figure}

In Table \ref{tab:L2_errors}, we present quantitative $L^2(0,R)$ errors of the recovered semilinear term $\widehat{F}(u)$. Fig.~\ref{Fig:Method2_1} shows the reconstructions in the full, half, and far-half data settings, and Fig.~\ref{Fig:Method2_2} shows the results for the quarter, far-quarter, and far-three-quarter settings. The full data case and the far three-quarter data case lead to the most accurate reconstructions, followed by the half data and quarter data cases, then followed by the far half data case, and the far quarter case gives the worst reconstructions. Specifically, for the far half data case and the far quarter data case, the recovered semilinear terms $\widehat F(u)$ are inaccurate.

\begin{table}[htb!]
\centering
\caption{The $L^2(0,R)$ errors for different measurement configurations and noise levels.}
\label{tab:L2_errors}
\begin{tabular}{l|cccc}
\hline
Configuration & $\varepsilon = 0$ & $\varepsilon = 1\%$ & $\varepsilon = 10\%$ & $\varepsilon = 20\%$ \\
\hline
Full data              & 1.71e-2 & 1.72e-2 & 2.08e-2 & 1.94e-2 \\
Half data ($x_1>0$)      & 2.78e-2 & 2.76e-2 & 2.74e-2 & 3.20e-2 \\
Far half data ($x_1<0$)  & 5.81e-2 & 5.80e-2 & 5.77e-2 & 5.90e-2 \\
Quarter data ($x_1>R/\sqrt{2}$)       & 3.86e-2 & 3.86e-2 & 3.70e-2 & 3.60e-2 \\
Far quarter data ($x_1<-R/\sqrt{2}$)  & 8.04e-2 & 8.04e-2 & 8.12e-2 & 8.04e-2 \\
Far three-quarter data ($x_1<R/\sqrt{2}$) & 1.40e-2 & 1.40e-2 & 1.43e-2 & 1.61e-2 \\
\hline
\end{tabular}
\end{table}

\begin{figure}[htb!]
\centering
\setlength{\tabcolsep}{0pt}
\begin{tabular}{cccc}
\includegraphics[width=0.25\linewidth]{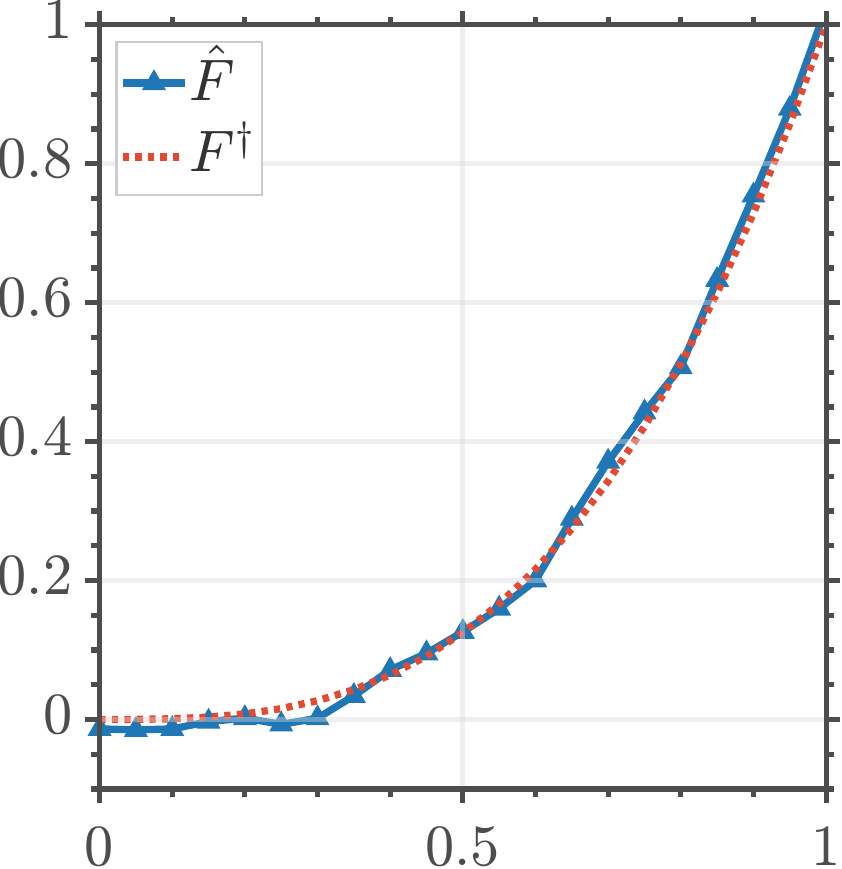}&
\includegraphics[width=0.25\linewidth]{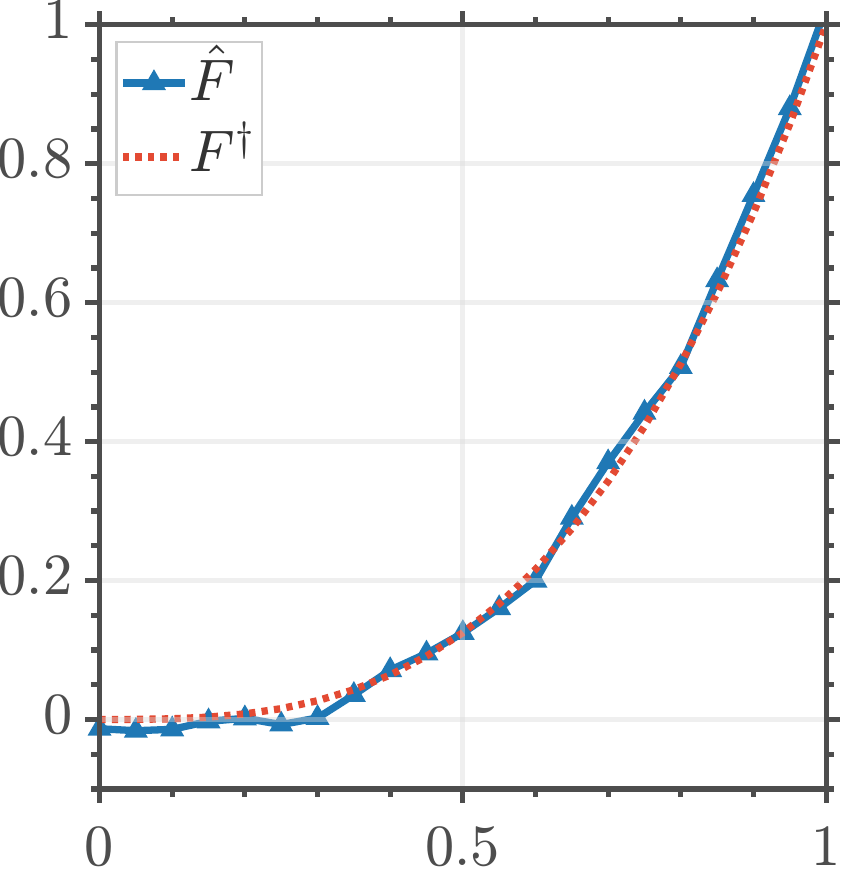}& 
\includegraphics[width=0.25\linewidth]{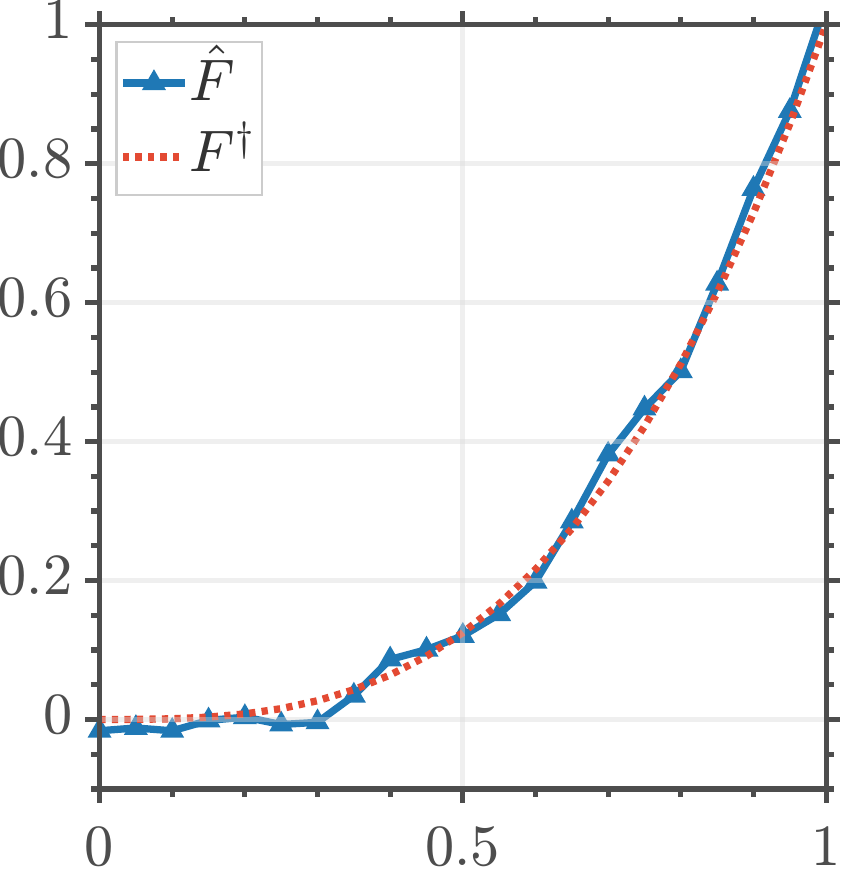}&
\includegraphics[width=0.25\linewidth]{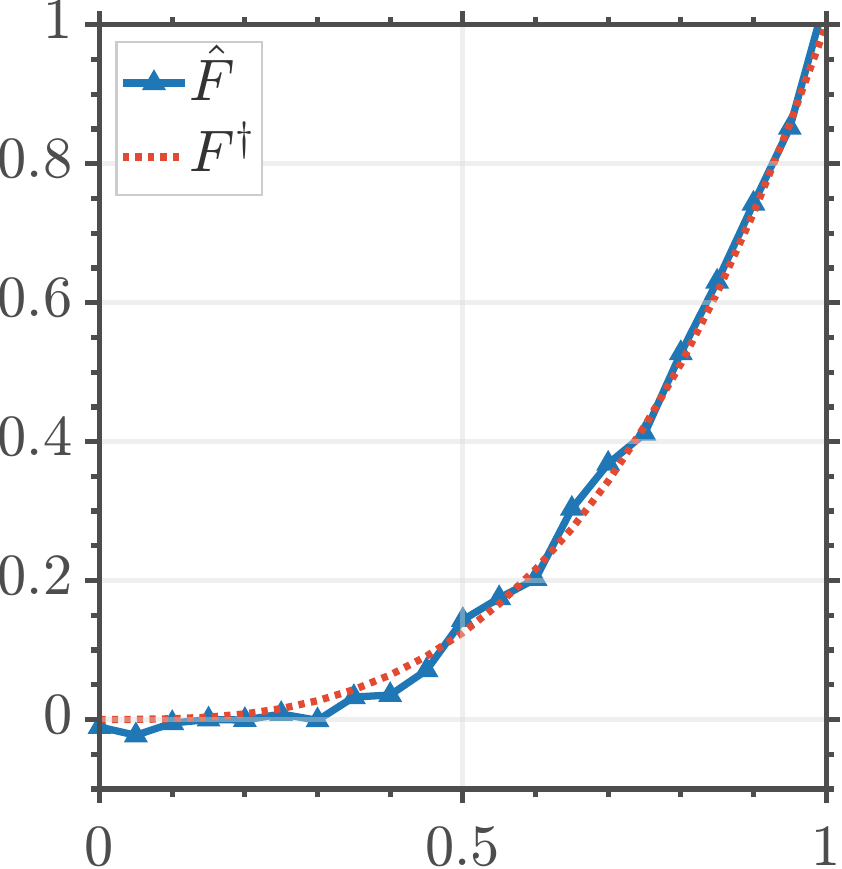}  \\
\includegraphics[width=0.25\linewidth]{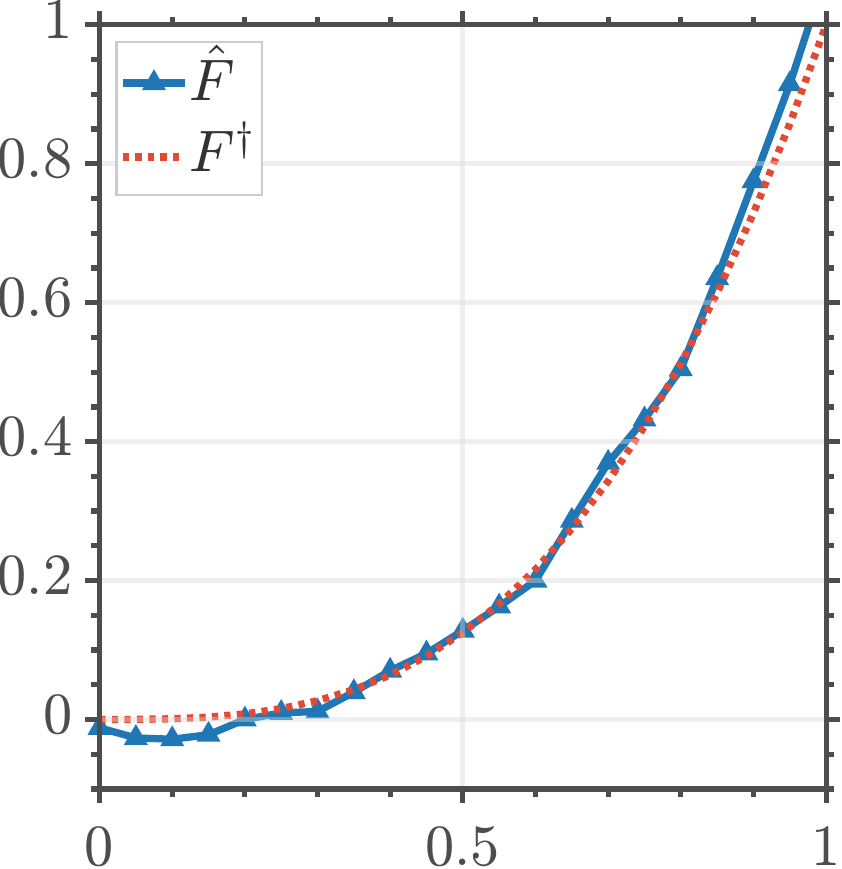}&
\includegraphics[width=0.25\linewidth]{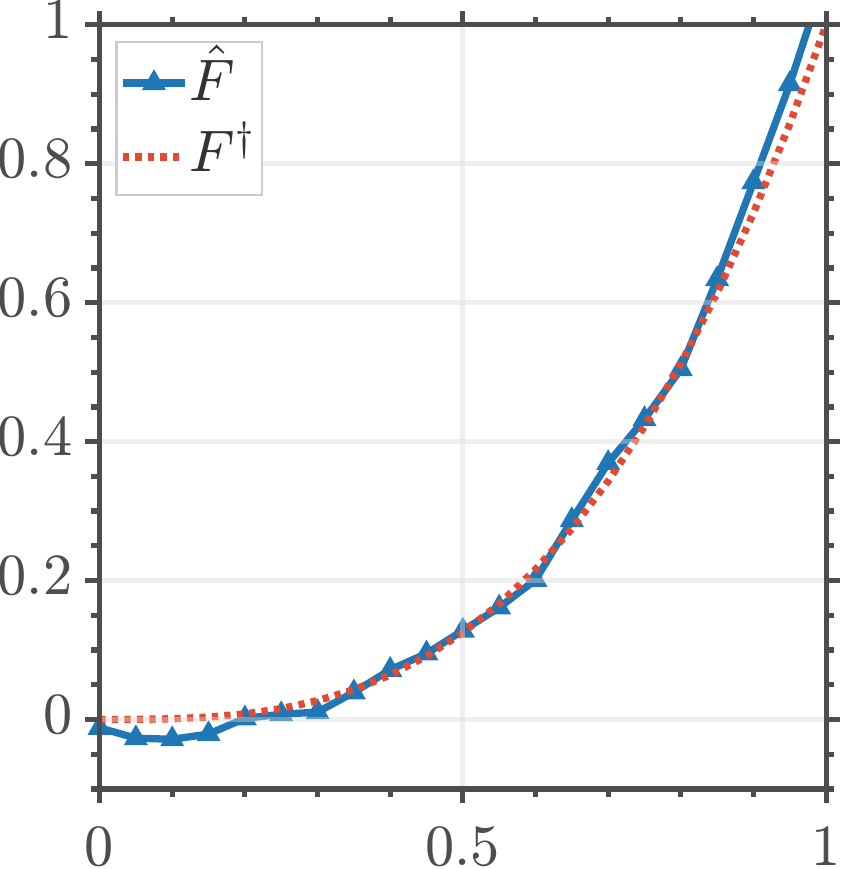}& 
\includegraphics[width=0.25\linewidth]{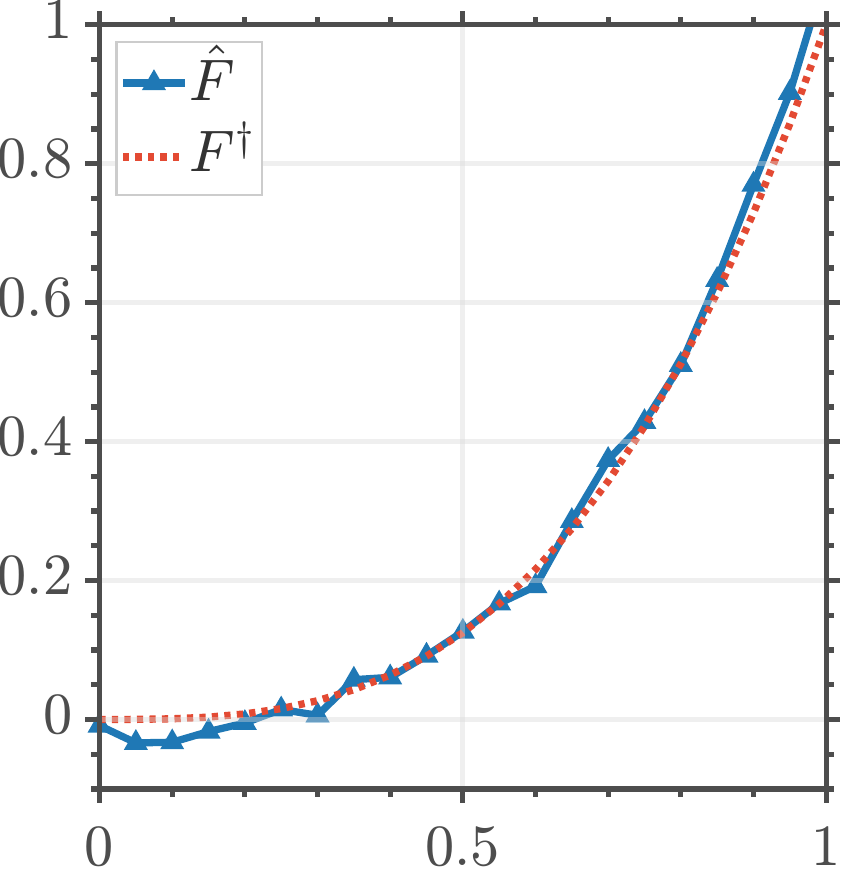}&
\includegraphics[width=0.25\linewidth]{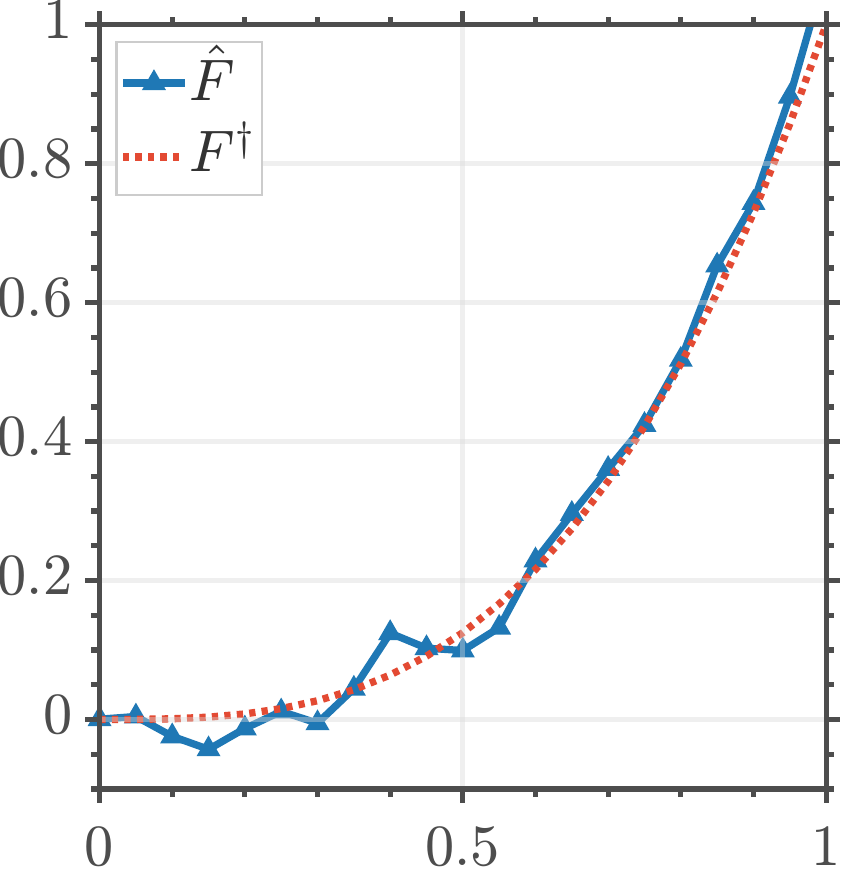}\\
\includegraphics[width=0.25\linewidth]{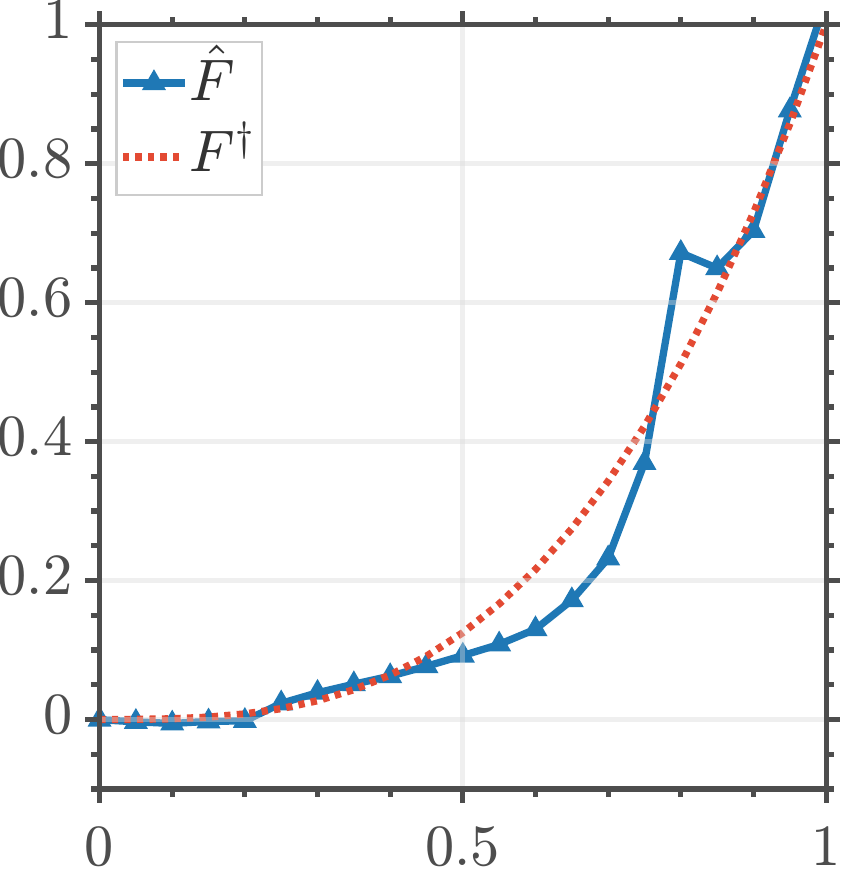}&
\includegraphics[width=0.25\linewidth]{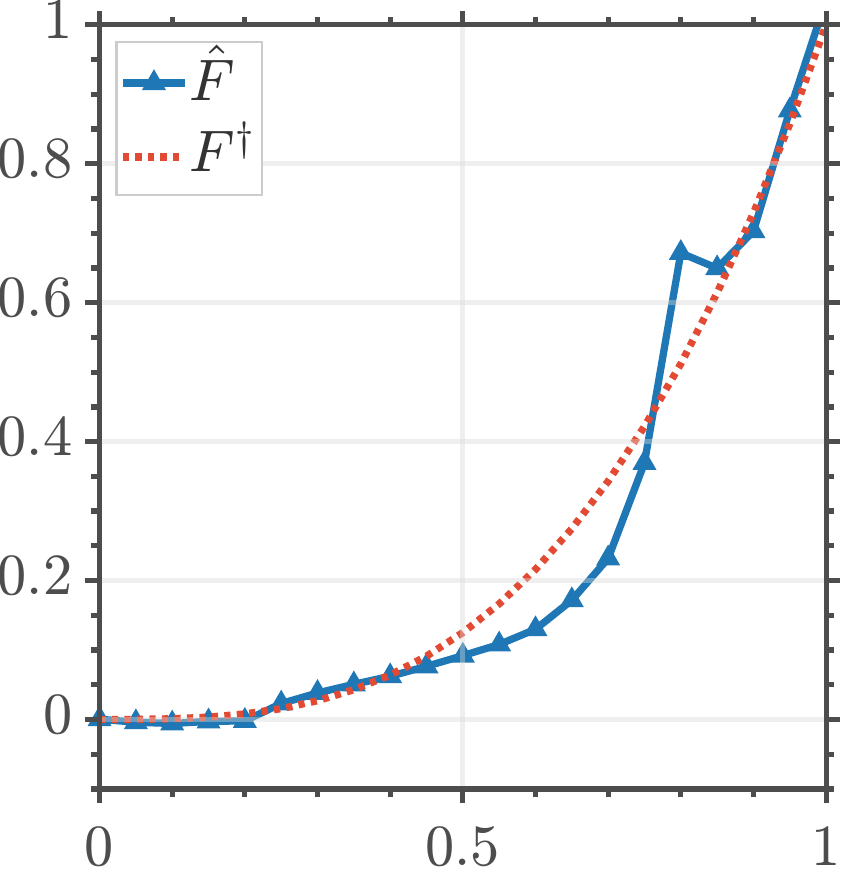}& 
\includegraphics[width=0.25\linewidth]{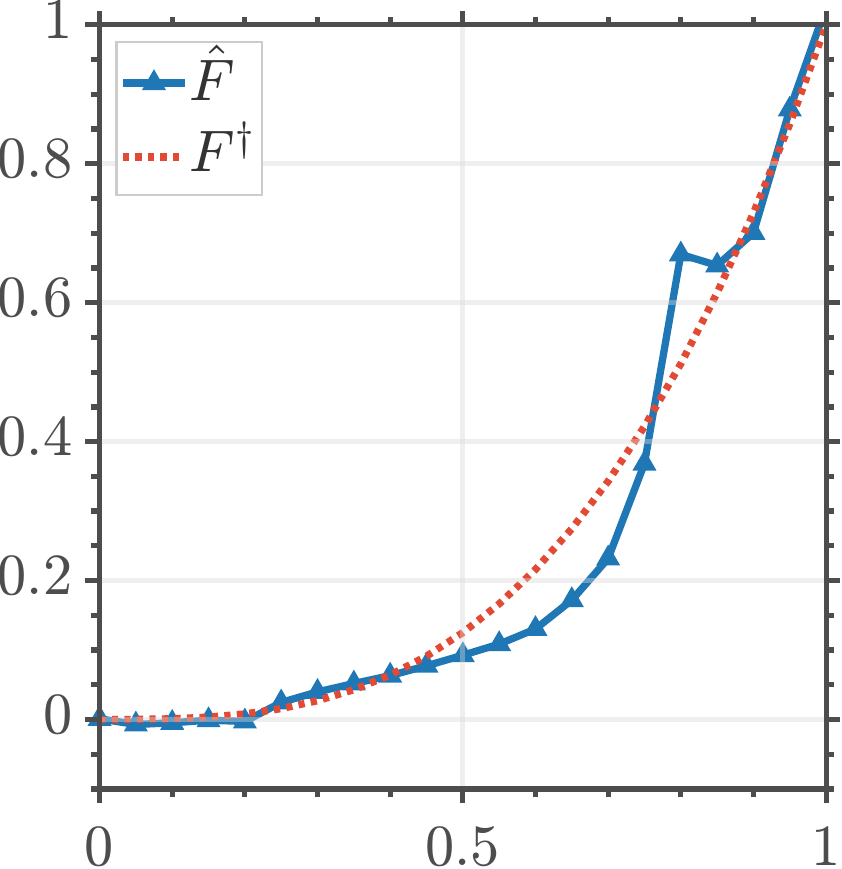}&
\includegraphics[width=0.25\linewidth]{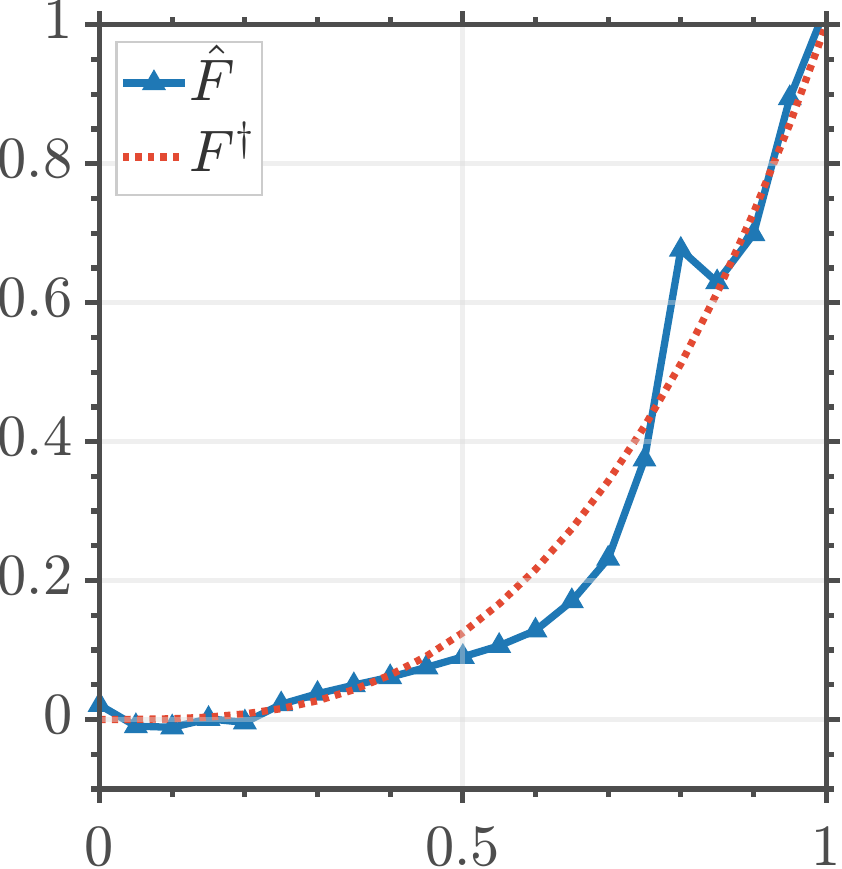}\\
$\varepsilon = 0$ & $\varepsilon = 1\%$ & $\varepsilon = 10\%$ & $\varepsilon = 20\%$
\end{tabular}
\caption{The numerical reconstructions for the cases at different noise levels: Full data, half data ($x_1>0$), far half data ($x_1<0$) (from top to bottom).}
\label{Fig:Method2_1}
\end{figure}

\begin{figure}[hbt!]
\centering
\setlength{\tabcolsep}{0pt}
\begin{tabular}{cccc}
\includegraphics[width=0.25\linewidth]{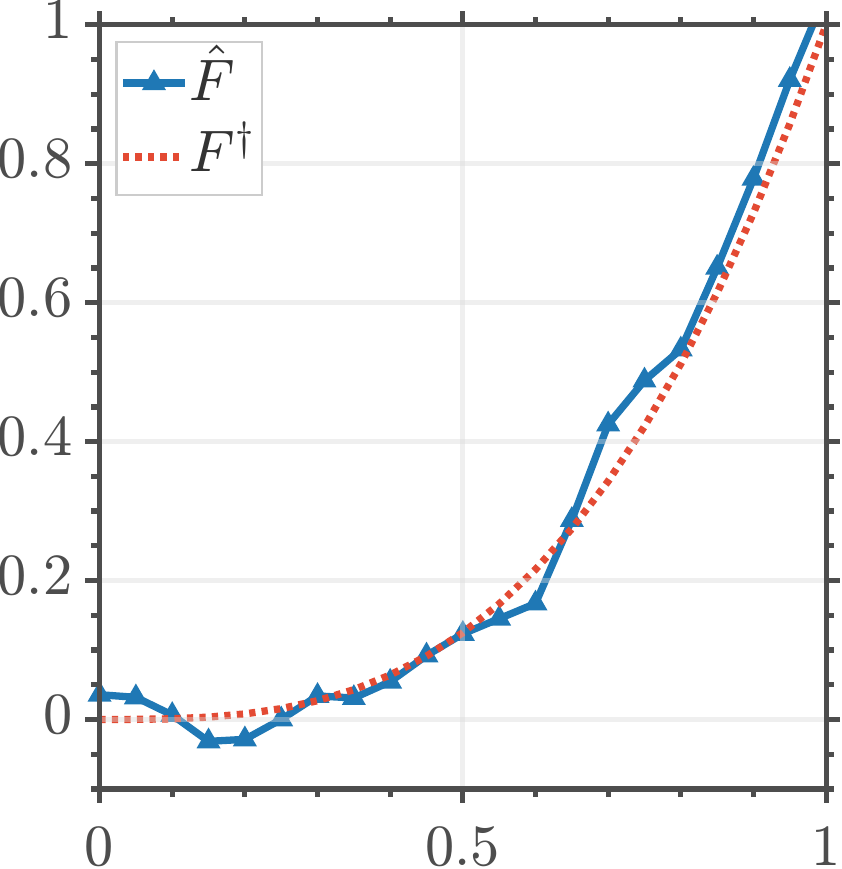}&
\includegraphics[width=0.25\linewidth]{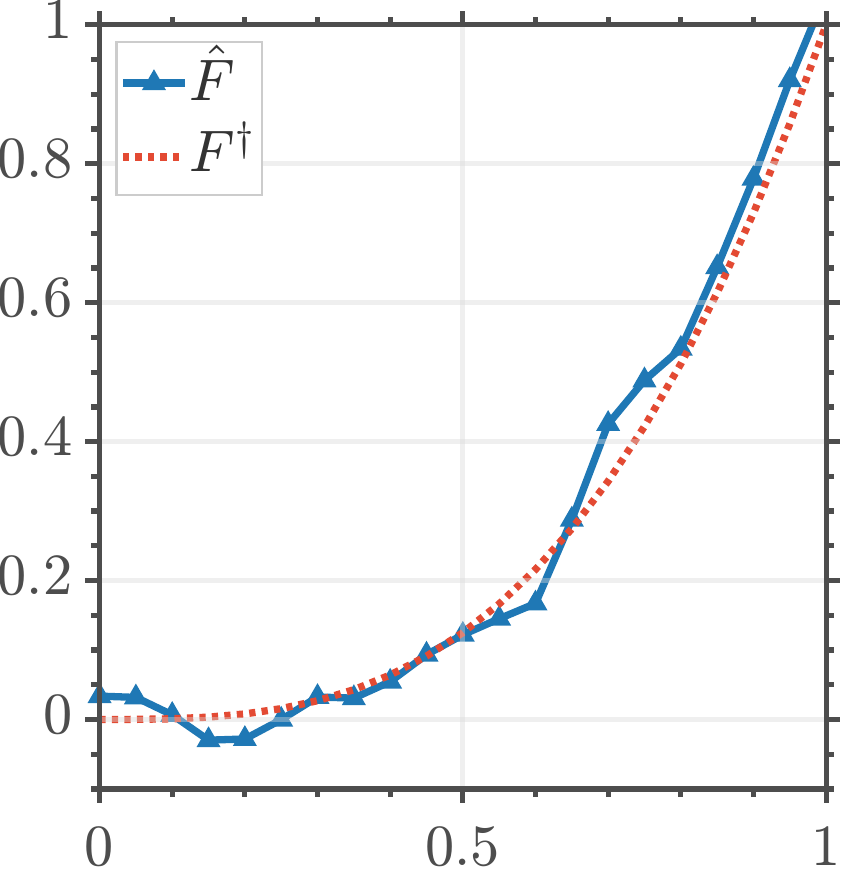}& 
\includegraphics[width=0.25\linewidth]{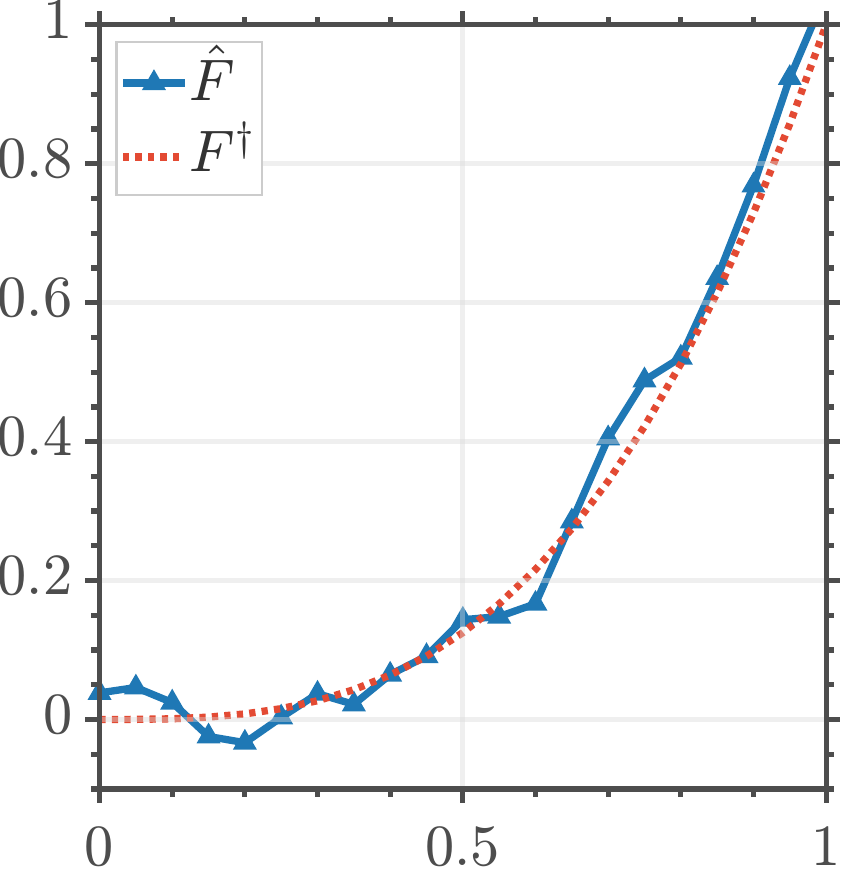}&
\includegraphics[width=0.25\linewidth]{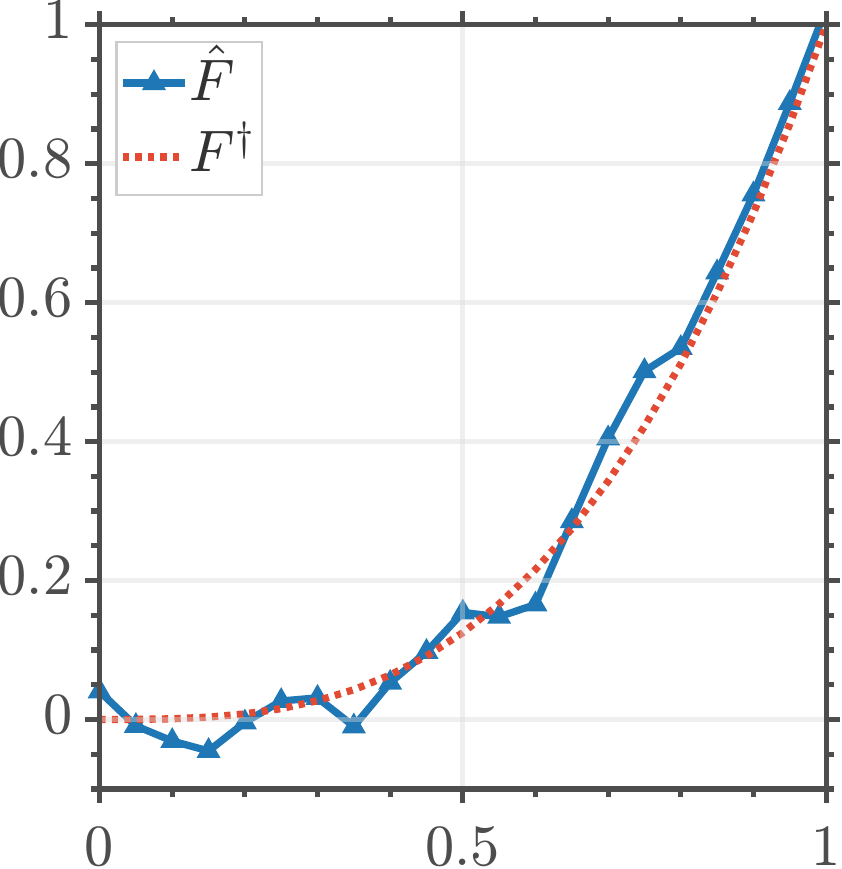}\\
\includegraphics[width=0.25\linewidth]{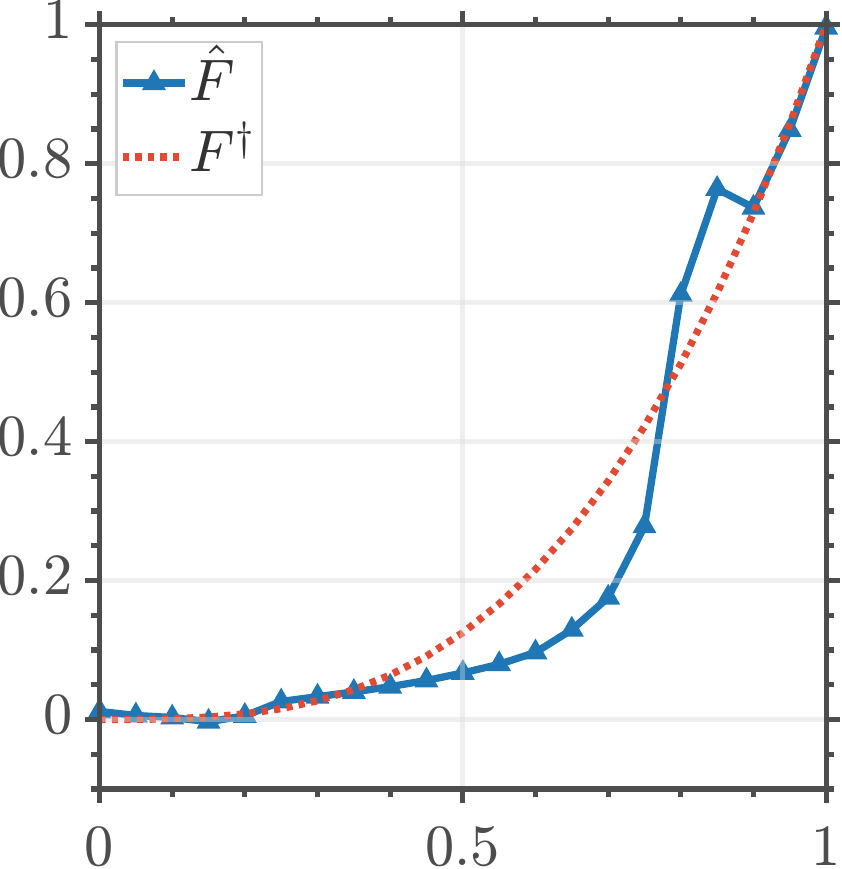}&
\includegraphics[width=0.25\linewidth]{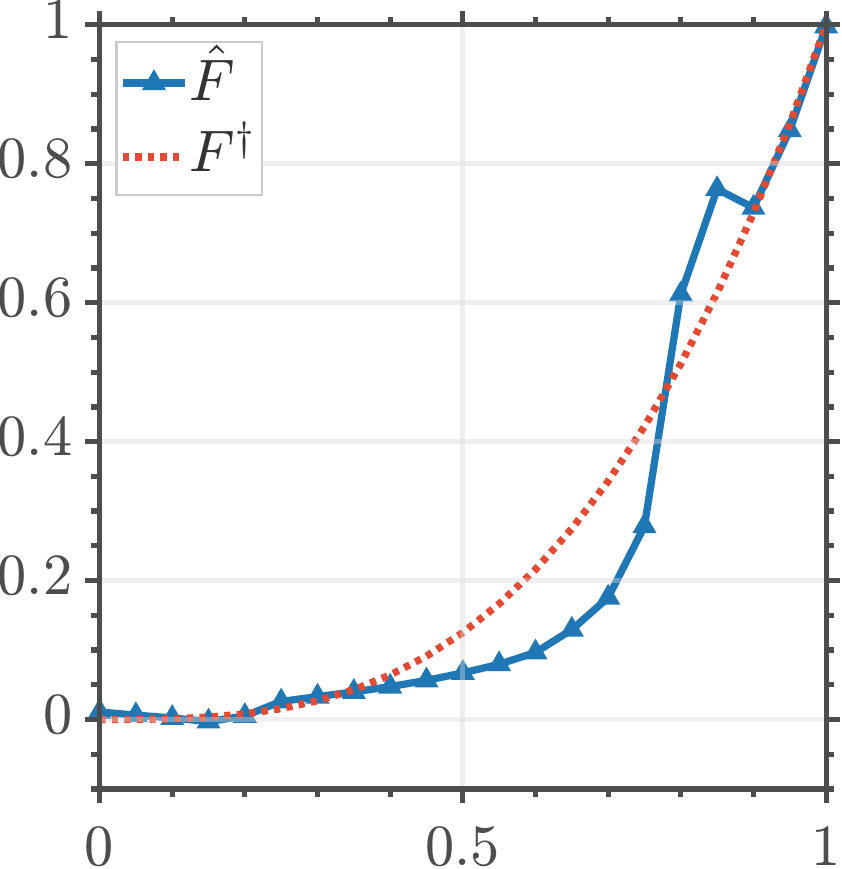}& 
\includegraphics[width=0.25\linewidth]{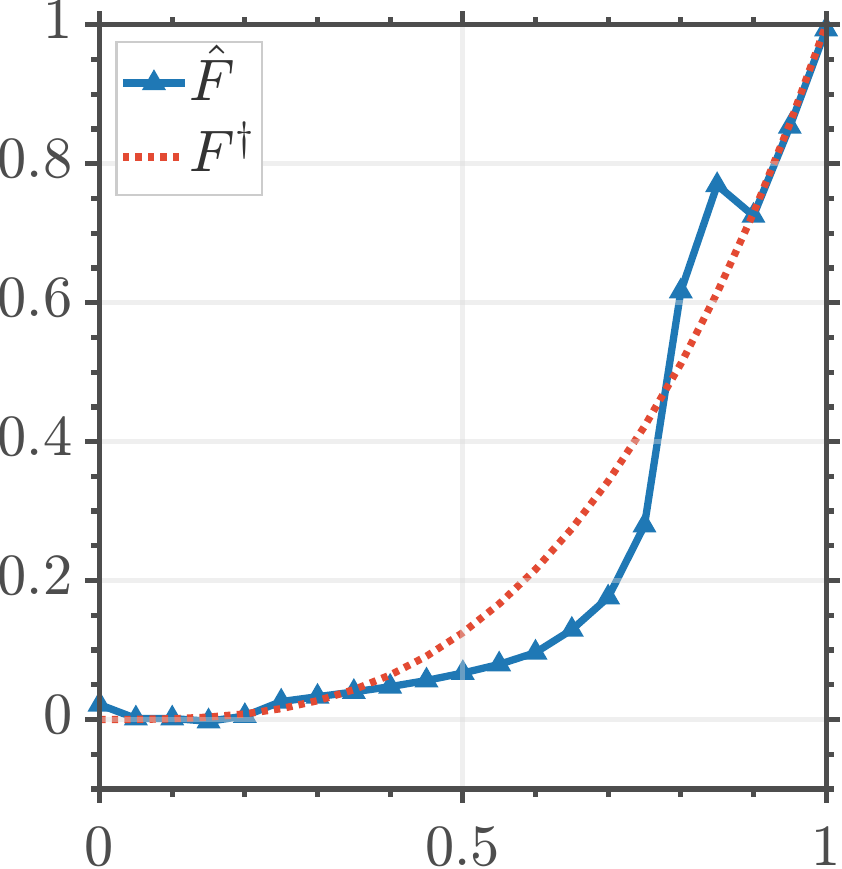}&
\includegraphics[width=0.25\linewidth]{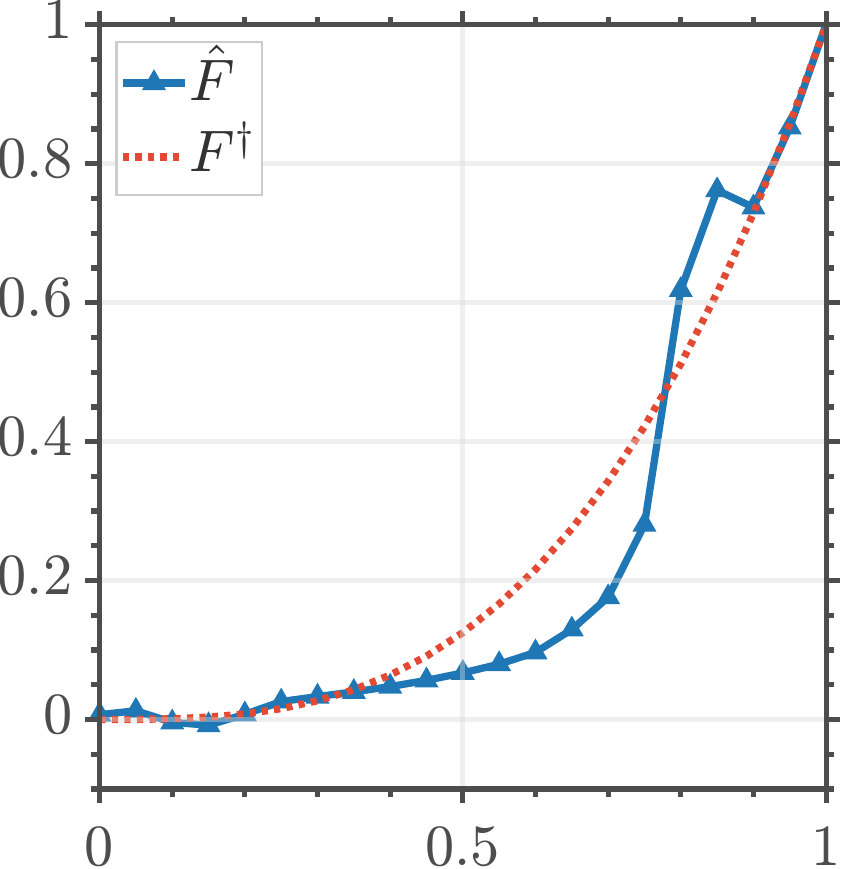}\\
\includegraphics[width=0.25\linewidth]{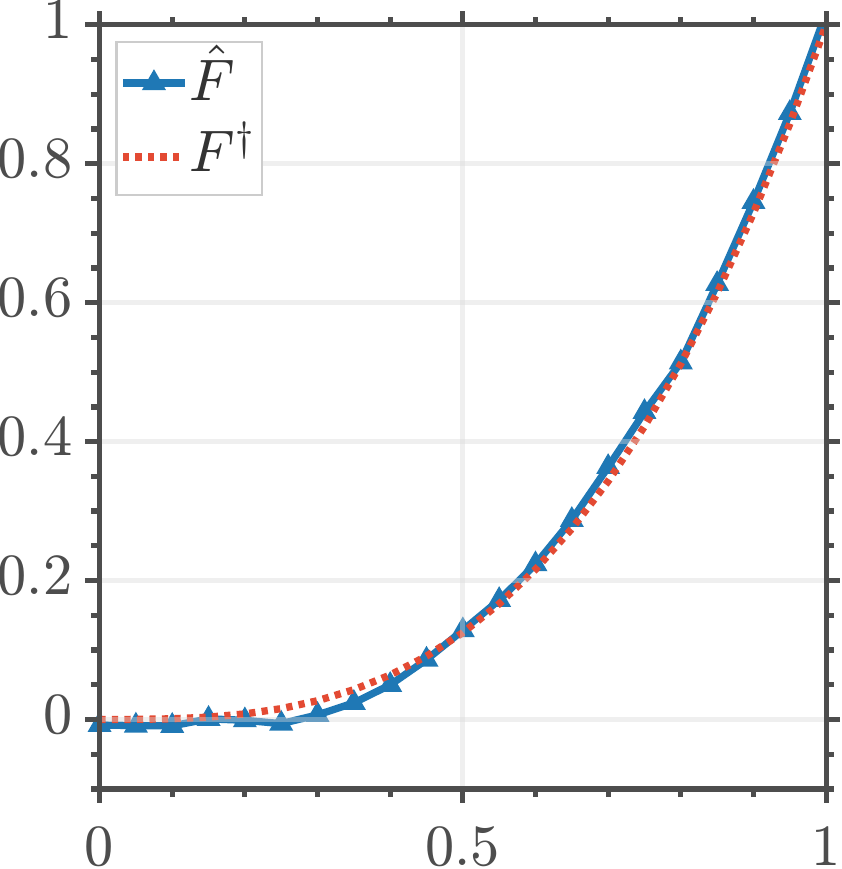}&
\includegraphics[width=0.25\linewidth]{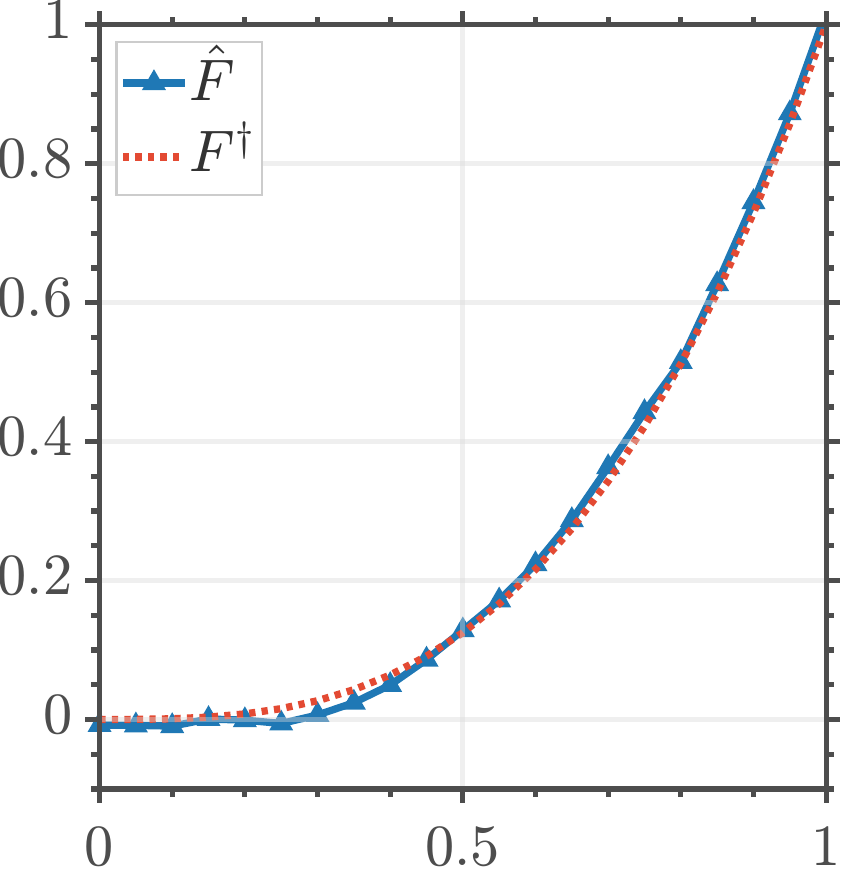}& 
\includegraphics[width=0.25\linewidth]{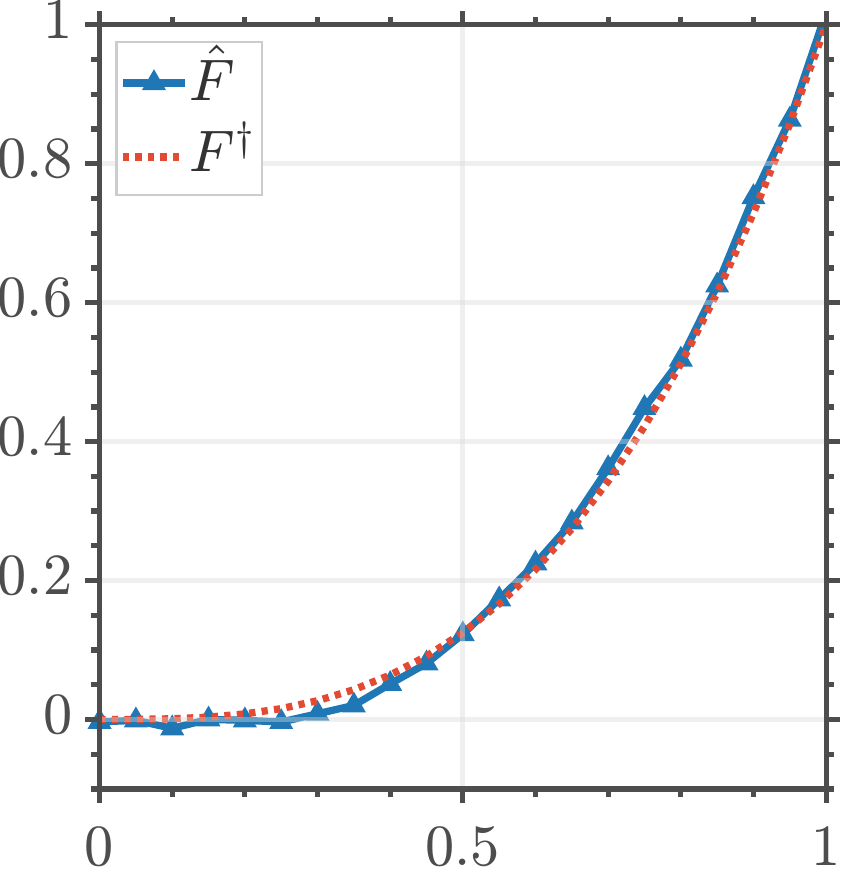}&
\includegraphics[width=0.25\linewidth]{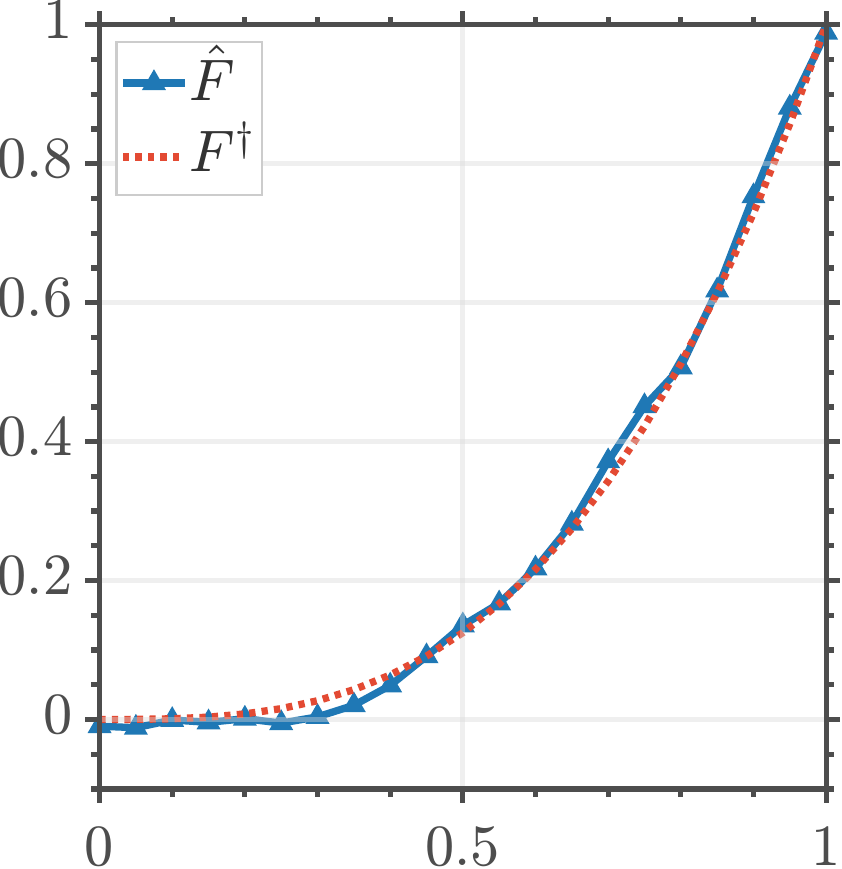}\\
$\varepsilon = 0$ & $\varepsilon = 1\%$ & $\varepsilon = 10\%$ & $\varepsilon = 20\%$
\end{tabular}
\caption{The numerical reconstructions for the cases at different noise levels: Quarter data ($x_1>R/\sqrt{2}$), far quarter data ($x_1<-R/\sqrt{2}$), far three-quarter data ($x_1<R/\sqrt{2}$) (from top to bottom).}    
\label{Fig:Method2_2}
\end{figure}

\bibliographystyle{abbrv} 
\bibliography{references}

@book{choulli2009introduction,
    AUTHOR = {Choulli, Mourad},
     TITLE = {Une introduction aux probl\`emes inverses elliptiques et
              paraboliques},
 PUBLISHER = {Springer-Verlag, Berlin},
      YEAR = {2009},
     PAGES = {xxii+249},
      ISBN = {978-3-642-02459-7},
   MRCLASS = {35-02 (35J15 35K10 35R30)},
  MRNUMBER = {2554831},
MRREVIEWER = {Paul\ E.\ Sacks},
       DOI = {10.1007/978-3-642-02460-3},
       URL = {https://doi.org/10.1007/978-3-642-02460-3},
}

@article{KDE,
 ISSN = {00034851},
 URL = {http://www.jstor.org/stable/2237880},
 author = {Emanuel Parzen},
 journal = { Annals Math. Stat.},
 number = {3},
 pages = {1065--1076},
 publisher = {Institute of Mathematical Statistics},
 title = {On Estimation of a Probability Density Function and Mode},
 urldate = {2026-05-06},
 volume = {33},
 year = {1962}
}

@article {CannonDuChateau:1998,
    AUTHOR = {Cannon, J. R. and DuChateau, Paul},
     TITLE = {Structural identification of an unknown source term in a heat
              equation},
   JOURNAL = {Inverse Problems},
  FJOURNAL = {Inverse Problems. An International Journal on the Theory and
              Practice of Inverse Problems, Inverse Methods and Computerized
              Inversion of Data},
    VOLUME = {14},
      YEAR = {1998},
    NUMBER = {3},
     PAGES = {535--551},
      ISSN = {0266-5611,1361-6420},
   MRCLASS = {35R30 (35K05 35K60 80A23)},
  MRNUMBER = {1629991},
MRREVIEWER = {Maurizio\ Grasselli},
       DOI = {10.1088/0266-5611/14/3/010},
       URL = {https://doi.org/10.1088/0266-5611/14/3/010},
}

@article{PR2,
 author = {Pilant, Michael and Rundell, William},
 title = {Fixed point methods for a nonlinear parabolic inverse coefficient problem},
 fjournal = {Communications in Partial Differential Equations},
 journal = {Commun. Partial Differ. Equations},
 issn = {0360-5302},
 volume = {13},
 number = {4},
 pages = {469--493},
 year = {1988},
 language = {English},
 doi = {10.1080/03605308808820549},
 zbMATH = {4056207},
 Zbl = {0647.35083}
}

@article{Lo,
 author = {Lorenzi, A.},
 title = {An inverse problem for a semilinear parabolic equation},
 fjournal = {Annali di Matematica Pura ed Applicata. Serie Quarta},
 journal = {Ann. Mat. Pura Appl. (4)},
 issn = {0373-3114},
 volume = {131},
 pages = {145--166},
 year = {1982},
 language = {English},
 doi = {10.1007/BF01765150},
 zbMATH = {3776102},
 Zbl = {0493.35078}
}

@book {ItoJin:2015,
    AUTHOR = {Ito, Kazufumi and Jin, Bangti},
     TITLE = {Inverse {P}roblems: {T}ikhonov {T}heory and {A}lgorithms},
 PUBLISHER = {World Scientific Publishing Co. Pte. Ltd., Hackensack, NJ},
      YEAR = {2015},
     PAGES = {x+318},
      ISBN = {978-981-4596-19-0},
   MRCLASS = {65-02 (49N45 65J22 65M32 65N21)},
  MRNUMBER = {3244283},
MRREVIEWER = {Sudhir H. Kulkarni},
}

@book{Benson:1960,
    author = {Benson, Sidney W.},
    title = {{The Foundations of Chemical Kinetics}},
    publisher = {McGraw-Hill, New York},
    year = {1960}
}

@book{Capasso:1998,
    author = {Capasso, Vincenzo},
    title = {{Mathematical Modelling for Polymer Processing}},
    publisher = {Springer, Berlin},
    year = {1998} 
}

@book{KrauseLoch:2002,
    editor = {Krause,  Dieter and Loch, Horst},
    title = {{Mathematical Simulation in Glass Technology}},
    publisher = {Springer, Berlin},
    year = {2002}
}

@article{Muzylev:1980,
    author = {Muzylev, N. V.},
    title = {Uniqueness theorems for some converse problems of heat equation},
    journal = {U.S.S.R. Compur. Muths. Murh. Phys.},
    year = {1980},
    volume = {20},
    number = {2},
     pages = {120--134},
}

@book {Thomee:2006,
    AUTHOR = {Thom\'{e}e, Vidar},
     TITLE = {Galerkin {F}inite {E}lement {M}ethods for {P}arabolic {P}roblems},
   EDITION = {Second},
 PUBLISHER = {Springer-Verlag, Berlin},
      YEAR = {2006},
     PAGES = {xii+370},
      ISBN = {978-3-540-33121-6; 3-540-33121-2},
   MRCLASS = {65-02 (65M15 65M60)},
  MRNUMBER = {2249024},
}

@article{PR,
 author = {Pilant, Michael  and Rundell, William},
 title = {An inverse problem for a nonlinear parabolic equation},
 fjournal = {Communications in Partial Differential Equations},
 journal = {Commun. Partial Differ. Equations},
 issn = {0360-5302},
 volume = {11},
 pages = {445--457},
 year = {1986},
 language = {English},
 doi = {10.1080/03605308608820430},
 zbMATH = {3955555},
 Zbl = {0594.35057}
}

@book{ladyzenskaja_linear_1968,
	address = {Providence, RI},
	title = {Linear and {Quasi}-linear {Equations} of {Parabolic} {Type}},
	publisher = {AMS},
	author = {Ladyzenskaja, O. and Solonnikov, V. and Uralceva, N.},
	translator = {Smith, S.},
	year = {1968},
	doi = {10.1090/mmono/023},
}

@article{choulli2006stable,
  title={Stable determination of a semilinear term in a parabolic equation},
  author={Choulli, Mourad and Ouhabaz, El Maati and Yamamoto, Masahiro},
  journal={Commun. Pure Appl. Anal.},
  volume={5},
  number={3},
  pages={447--462},
  year={2006},
  publisher={AIM SCIENCES}
}

@book{friedman2008partial,
    AUTHOR = {Friedman, Avner},
     TITLE = {{Partial Differential Equations of Parabolic Type}},
 PUBLISHER = {Prentice-Hall, Inc., Englewood Cliffs, NJ},
      YEAR = {1964},
     PAGES = {xiv+347},
   MRCLASS = {35.00 (35.62)},
  MRNUMBER = {181836},
MRREVIEWER = {B.\ Frank\ Jones, Jr.},
}

@book{ito1992diffusion,
    AUTHOR = {It\^o, Seiz\^o},
     TITLE = {{Diffusion Equations}},
 PUBLISHER = {AMS, Providence, RI},
      YEAR = {1992},
     PAGES = {x+225},
      ISBN = {0-8218-4570-5},
   MRCLASS = {35K10 (35-01 35-02 47N20)},
  MRNUMBER = {1195786},
MRREVIEWER = {P.\ Szeptycki},
       DOI = {10.1090/mmono/114},
       URL = {https://doi.org/10.1090/mmono/114},
}

@article{isakov1993uniqueness,
  title={On uniqueness in inverse problems for semilinear parabolic equations},
  author={Isakov, Victor},
  journal={Arch. Ration. Mech. Anal.},
  volume={124},
  number={1},
  pages={1--12},
  year={1993},
  publisher={Springer}
}

@article{DR,
 author = {DuChateau, Paul and Rundell, William},
 title = {Unicity in an inverse problem for an unknown reaction term in a reaction- diffusion equation},
 journal = {J. Differ. Equations},
 issn = {0022-0396},
 volume = {59},
 pages = {155--164},
 year = {1985},
 language = {English},
 doi = {10.1016/0022-0396(85)90152-4},
 zbMATH = {3899457},
 Zbl = {0564.35097}
}

@article{SU,
 author = {Sun, Ziqi and Uhlmann, Gunther},
 title = {Inverse problems in quasilinear anisotropic media},
 fjournal = {American Journal of Mathematics},
 journal = {Am. J. Math.},
 issn = {0002-9327},
 volume = {119},
 number = {4},
 pages = {771--797},
 year = {1997},
 language = {English},
 doi = {10.1353/ajm.1997.0027},
 url = {muse.jhu.edu/journals/american_journal_of_mathematics/toc/ajm119.4.html},
 zbMATH = {1052803},
 Zbl = {0886.35176}
}

@article{KR1,
 author = {Kaltenbacher, Barbara and Rundell, William},
 title = {On the identification of a nonlinear term in a reaction-diffusion equation},
 fjournal = {Inverse Problems},
 journal = {Inverse Problems},
 issn = {0266-5611},
 volume = {35},
 number = {11},
 pages = {115007, 38 pp.},
 year = {2019},
 language = {English},
 doi = {10.1088/1361-6420/ab2aab},
 zbMATH = {7115209},
 Zbl = {1427.35359}
}

@article{KR2,
 author = {Kaltenbacher, Barbara and Rundell, William},
 title = {Recovery of multiple coefficients in a reaction-diffusion equation},
 fjournal = {Journal of Mathematical Analysis and Applications},
 journal = {J. Math. Anal. Appl.},
 issn = {0022-247X},
 volume = {481},
 number = {1},
 pages = {123475, 23 pp.},
 year = {2020},
 language = {English},
 doi = {10.1016/j.jmaa.2019.123475},
 zbMATH = {7113650},
 Zbl = {1427.35358}
}

@article{CZ,
 author = {Choulli, Mourad and Zeghal, Ahmed},
 title = {Un r{\'e}sultat d'unicit{\'e} pour un probl{\`e}me inverse parabolique semi-lin{\'e}aire. ({A} uniqueness result for a semilinear parabolic inverse problem)},
 fjournal = {Comptes Rendus de l'Acad{\'e}mie des Sciences. S{\'e}rie I},
 journal = {C. R. Acad. Sci., Paris, S{\'e}r. I},
 issn = {0764-4442},
 volume = {315},
 number = {10},
 pages = {1051--1053},
 year = {1992},
 language = {French},
 zbMATH = {148780},
 Zbl = {0763.35104}
}

@article{KLU,
 author = {Kurylev, Yaroslav and Lassas, Matti and Uhlmann, Gunther},
 title = {Inverse problems for {Lorentzian} manifolds and non-linear hyperbolic equations},
 fjournal = {Inventiones Mathematicae},
 journal = {Invent. Math.},
 issn = {0020-9910},
 volume = {212},
 number = {3},
 pages = {781--857},
 year = {2018},
 language = {English},
 doi = {10.1007/s00222-017-0780-y},
 zbMATH = {6897402},
 Zbl = {1396.35074}
}

@article{kian2024determining,
    AUTHOR = {Kian, Yavar and Liimatainen, Tony and Lin, Yi-Hsuan},
     TITLE = {On determining and breaking the gauge class in inverse
              problems for reaction-diffusion equations},
   JOURNAL = {Forum Math. Sigma},
  FJOURNAL = {Forum of Mathematics. Sigma},
    VOLUME = {12},
      YEAR = {2024},
     PAGES = {e25, 42 pp.},
      ISSN = {2050-5094},
   MRCLASS = {35R30 (35K57)},
  MRNUMBER = {4710715},
       DOI = {10.1017/fms.2024.18},
       URL = {https://doi.org/10.1017/fms.2024.18},
}

@article{homberg2019uniqueness,
  title={Uniqueness for an inverse problem for a nonlinear parabolic system with an integral term by one-point Dirichlet data},
  author={H{\"o}mberg, Dietmar and Lu, Shuai and Yamamoto, Masahiro},
  journal={J. Differ. Equations},
  volume={266},
  number={11},
  pages={7525--7544},
  year={2019},
  publisher={Elsevier}
}

@article{choy2026simultaneous,
  title={Simultaneous stable determination of quasilinear terms for parabolic equations},
  author={Choy, Jason and Kian, Yavar},
  journal={Nonlinear Anal. Real World Appl.},
  volume={87},
  pages={104442},
  year={2026},
  publisher={Elsevier}
}

@article{chadam1994diffusion,
  title={A diffusion equation with localized chemical reactions},
  author={Chadam, John M and Yin, Hong-Ming},
  journal = {Proc. Edinb. Math. Soc.},
  fjournal={Proceedings of the Edinburgh Mathematical Society},
  volume={37},
  number={1},
  pages={101--118},
  year={1994},
  publisher={Cambridge University Press}
}

@article{Yoshizawa_1970, title={Population growth process described by a semilinear parabolic equation}, volume={7}, ISSN={0025-5564}, DOI={https://doi.org/10.1016/0025-5564(70)90129-X}, number={3}, journal={Math. Biosci.}, author={Yoshizawa, Shuji}, year={1970}, pages={291–303} }

@article{egger2005global,
  title={Global uniqueness and {H}{\"o}lder stability for recovering a nonlinear source term in a parabolic equation},
  author={Egger, Herbert and Engl, Heinz W and Klibanov, Michael V},
  journal={Inverse problems},
  volume={21},
  number={1},
  pages={271--290},
  year={2005}
}

@article{feizmohammadi2022inverse,
  title={An inverse problem for a quasilinear convection--diffusion equation},
  author={Feizmohammadi, Ali and Kian, Yavar and Uhlmann, Gunther},
  journal={Nonlin. Anal.},
  volume={222},
  pages={112921},
  year={2022},
  publisher={Elsevier}
}

@article{Kian_Uhlmann_2023, title={Recovery of Nonlinear Terms for Reaction Diffusion Equations from Boundary Measurements}, volume={247}, ISSN={1432-0673}, DOI={10.1007/s00205-022-01831-y},  
 number={1}, journal={Arch. Ration. Mech. Anal.}, 
 author={Kian, Yavar and Uhlmann, Gunther}, 
 year={2023}, 
 pages={6, 20 pp.}, 
 }

@article{isakov2001uniqueness,
  title={Uniqueness of recovery of some systems of semilinear partial differential equations},
  author={Isakov, Victor},
  journal={Inverse Problems},
  volume={17},
  number={4},
  pages={607--618},
  year={2001}
}

@misc{kian2024determination,
  title={Determination and reconstruction of a semilinear term from point measurements},
  author={Kian, Yavar and Liu, Hongyu and Wang, Li-Li and Zheng, Guang-Hui},
  howpublished ={Preprint, arXiv:2411.09922},
  year={2024}
}
\end{document}